\documentclass[11pt,reqno]{amsart}

\usepackage{amsfonts} 
\usepackage{amsmath,amscd}
\usepackage{amssymb}
\usepackage{tikz-cd}
\usepackage{amsthm}
\usepackage{cases}

\usepackage{galois}
\usepackage{chemarrow}
\usepackage[all]{xy}
\usepackage{graphicx}
\usepackage[colorlinks,
            linkcolor=red,
            anchorcolor=black,
            citecolor=blue
            ]{hyperref}	
\usepackage{mathrsfs}
\usepackage{extarrows}
\usepackage{texnames}
\usepackage{textcase}

\makeatletter
\def\section{%
  \@startsection{section}{1}
    {\z@}
    {2.0ex plus 0.8ex minus .1ex}
    {1.0ex plus .2ex}
    {\large\bfseries\boldmath\centering\MakeTextUppercase}%
}
\makeatother
\newcommand{\subsub}{\subset \subset}
\newcommand{\up}[1]{{ {}^{#1} } }
\newcommand{\gradh} { {{}^h \nabla }}
\newcommand{\gradg}{\nabla}
\newcommand{\gradhath} { {{}^{\hat h} \nabla }}
\newcommand{\la}{\lambda}

\def\curlS{\mathcal S}
\def\curlX{\mathcal X}
\def\curlN{\mathcal N}
\def\curlt{\mathfrak t}
\newcommand{\E}{ \mathbb E}

\def\curlM{\mathcal M}
\def\curlR{\mathcal R}
\def\curlP{\mathcal P}

\def\curlT{\mathcal T}

\def\curlC{\mathcal C}
\def\grad{\nabla}

\newcommand{\si}{\sigma}
\newcommand{\downto}{\searrow}
\newcommand{\ga}{\gamma}

\newcommand{\ep}{\epsilon}

\newcommand{\al}{\alpha}

\newcommand{\vol}{V}

\newcommand{\lap}{\Delta}
\newcommand{\de}{\delta}
\newcommand{\be}{\beta}

\newcommand{\ti}{\tilde}
\def\of{\circ}
\newcommand{\partt}{\frac{\partial}{\partial t}}

\def\Sc{{\rm  R}}
\def\Rm{{\rm  Rm}}
\def\si{\sigma}
\def\Riem{{\rm  Rm}}
\def\Ric{{\rm Ric}}
\def\Rc{{\rm Ric}}
\def\Ricci{{\rm Ric}}

\newcommand{\N}{ \mathbb N}
\newcommand{\R}{ \mathbb R}
\newcommand{\B}{ \mathbb B}

\newtheorem{thm}{Theorem}[section]
\newtheorem{lem}[thm]{Lemma}

\newtheorem{rem}[thm]{Remark}

\newtheorem{defi}[thm]{Definition}

\numberwithin{equation}{section}

\begin{document}
\title[Preserving curvature lower bounds under Ricci flow of non-smooth manifolds]{Preserving curvature lower bounds when  Ricci flowing   non-smooth initial data }
\author[Miles Simon]{Miles Simon}

\address{Miles Simon\\ Institut f\"ur Analysis und Numerik\\ Otto-von-Guericke-Universit\"at Magdeburg\\ universit\"atsplatz 2\\ Magdeburg 39106\\ Germany\\} \email{miles.simon@ovgu.de}
\thanks{The author is  supported by the Special Priority Program SPP 2026 ``Geometry at Infinity" from the German Research Foundation (DFG)}
\maketitle
\vskip -3.99ex

\centerline{\noindent\mbox{\rule{3.99cm}{0.5pt}}}

\vskip 5.01ex

\ \ \ \ {\bf Abstract:} In this paper we survey some results on Ricci flowing non-smooth initial data.    Among other things, we give a non-exhaustive list of various  weak initial data which can be evolved with the Ricci flow.  We also survey results which show that various  curvature  lower bounds will, possibly up to a constant,  be preserved, if we start with   such possibly non-smooth initial  data. Some   proofs/proof sketches  are given in certain cases.
A list of some open problems related to these areas is given in the last section of the  paper.

\medskip


\section{Introduction}\label{sec:intro}
 For a given smooth manifold $M$,  and an interval $I,$ a family of smooth Riemannian metrics  $(M,g(t))_{t\in I}$ is a solution to Ricci flow 
if $$\partt g(t)= -2\Rc(g(t)) \text{ for all }  t\in I .$$
In the paper \cite{ham3d} Hamilton introduced the Ricci flow   and used  it to prove that any smooth closed three dimensional manifold with  positive Ricci curvature  must be a space form. 
An important property of the Ricci flow which Hamilton  and later others used, is that the non-negativity of various curvature quantities  and in the case of scalar  curvature lower bounds are   preserved under the flow.
Here we present a non-exhaustive list of some curvature    lower bounds  which are preserved under the Ricci flow (definitions of and a brief introduction to the curvature conditions considered here are  given in the next section, Section \ref{flows-and-curvature}).
Here $\Sc(g)$ denotes the scalar curvature, $\Ricci(g)$ the Ricci curvature of $g$ and $\sec(g)\geq k$ means that the sectional curvatures of $g$ are everywhere larger than $k$.\\
In this paper, we use the term {\it manifold} for a manifold without boundary.
Let $(M^n,g(t))_{t\in [0,T]}$ be a smooth, complete, connected  solution   to Ricci flow, and if $M^n$ is non-compact assume further that   
 $\sup_{M\times [0,T]} |\Rm(g(t))|< \infty$ (we call a  solution $(M,g(t))_{t\in I},$ where $I$ is an interval,  {\bf a bounded curvature solution}, if $\sup_{M \times [a,b]} |\Rm(g)|< \infty$ for all compact $[a,b] \subseteq I$ ).  Then 
\begin{itemize} 
\item[(i)]
$M=M^3$    and $\Rc(g(0))\geq 0 \ (\Rc(g(0)) >0)     \implies \Rc(g(t))\geq 0 \ (\Rc(g(t))> 0) $ for all $t\in [0,T]$: Hamilton :  \cite{ham3d}
\item[(ii)] $M=M^3$  and $\sec(g(0)) \geq  0 \ ( \sec(g(0)) > 0 )  $ 
$\implies$ $\sec(g(t)) \geq   0  \ ( \sec(g(t)) >  0 ) $ for all $t\in [0,T].$
Hamilton :  \cite{HamFor}
\item[(iii)] 
$M=M^3,$ $\Sc(g(0))>0$    and $\Rc(g(0))\geq\ep\Sc(g(0)) $ $\implies$ $\Sc(g(t))>0$    and $\Rc(g(t)) \geq\ep\Sc(g(t)) $  for all $t\in [0,T]$
: Hamilton  \cite{ham3d} (Lott  \cite{Lott-Ricci-pinched}). 
\item[(iv)]
$M=M^n,$ $k \in \R$  and $\Sc(g(0)) \geq k$ $\implies$ $\Sc(g(t))\geq k$ for all $t\in [0,T]$ . Hamilton,  \cite{ham3d}.
\item[(v)]
$(M^n ,g(0))$  has  non-negative  isotropic curvature   $\implies$
$(M^n ,g(t))$  has   non-negative  isotropic curvature    for all $t\in [0,T]$: Hamilton, $n=4$, \cite{HamIso},  
Brendle/Schoen \cite{Brendle-Schoen} and  Ngyuen \cite{HuyNguyenThesis} $n \in \N$. 
\item[(vi)] $(M^n ,g(0))$ has non-negative (positive)  curvature
operator $\implies $  $(M^n ,g(t))$ has non-negative (positive)  curvature operator for all $t\in [0,T]$, Hamilton,    \cite{ham4d}
\item[(vii)]   $(M^n ,g(0))$ has $2$-non-negative (positive)  curvature
operator $\implies $  $(M^n ,g(t))$ has $2$-non-negative (positive)  curvature operator for all $t\in [0,T].$  Hamilton, \cite{HamFor}
\item[(viii)]  $(M^n ,g(0))$ is K\"ahler and has non-negative    holomorphic  bisectional curvature, then
  $(M^n ,g(t))$ is K\"ahler and  has non-negative    holomorphic  bisectional curvature, Mok  \cite{Mok}, ( Bando \cite{Bando} for the three dimensional case).
\end{itemize}

\begin{rem}
Regarding (i) and (ii): There are examples of closed smooth  Riemannian four manifolds   $(M^4,g(0)) $ with  $\Ric(g(0))>0$  ( $\sec(g(0))>0$) such that the solution  $(M^4,g(t))_{t\in [0,T)} $ to Ricci flow does {\bf not} have $\Ric(g(t))\geq 0$
( $\sec(g(t))\geq 0$  ) everywhere on $M$ for all $t\in [0,T).$  See for example the papers of {\rm Maximo}, \cite{Maximo}, respectively {\rm Bettiol}  and {\rm Krishnan}  \cite{BetKris}). 
\end{rem}
\begin{rem}
We survey   results for manifolds without boundary in this paper. However, there are
a number of result on preserving curvature conditions and Ricci flowing non-smooth data  for  manifolds with boundary : see
for example  \cite{PulBound}, \cite{GiaBound} , 
\cite{TCBound},  \cite{PhDschlichting},  \cite{BreBound}, \cite{CorMurBound}, \cite{She96} (and the references therein) for definitions of and theorems on  Ricci flow with boundary,  preserving curvature conditions / Ricci flowing non-smooth data  under  Ricci flow with boundary.   
\end{rem}

In Metric  and Differential Geometry an important role is often played  by spaces which are not necessarily smooth but have some curvature bounds from below or above  in a weak sense.
These spaces sometimes  occur as limits of smooth spaces with the same or similar curvature bounds, sometimes they are constructed  directly. 

Some famous examples of classes  of such metric spaces with curvature bounded from below are:  

The spaces   {$\curlM(n,k) $ } of Alexandrov et al. : {\it Alexandrov   metric spaces with dimension $n$ and sectional curvature  not less than $k$} : see \cite{Bur-Gro-Per}\\
The {\it $RCD(k,n)$ ($RCD^{*}(k,n)$) Spaces} introduced and studied  by Lott/Villani  and  Sturm
:  {\it Metric spaces with dimension not larger  than  $n$ and Ricci  curvature  not less than $k$} : see the surveys \cite{Gig} and \cite{Cav-Mil}.

Such spaces may have singularities, either of the metric type or in the topological sense.
A metric space $(M,d)$ is {\it smooth at} a point  $p\in M$ if there is a neighbourhood $U$ of $p$ in $M$ and a homeomorphism $\phi: U \subseteq M  \to   \B_{r}(0) \subseteq \R^n,$ and a smooth Riemannian metric $g$ on $ \B_{r}(0)$ such that the push forward  $\ti d$ of $d$, $\ti d(x,y):= d(\phi^{-1}(x), \phi^{-1}(y))$ agrees with the metric $\hat d$ induced by $(\B_{r}(0),g)$ on $\B_{s}(0)$ for some $s>0$.\\
A metric space $(M,d)$ has a {\it topological singularity} at $p\in M$ if there is no neighbourhood of $p$ in $M$ which is homeomorphic to an  euclidean ball.\\
A metric space $(M,d)$ has a {\it metric singularity} at $p\in M$ if it has no topological singularity at $p\in M$ but it is not smooth at $p$.

The Ricci flow has also been successfully used to evolve   many non-smooth metric spaces with curvature bounds from below in some weak sense and this has led to many applications in differential  and metric geometry. 
\begin{defi}
Let $(M,g(t))_{t \in (0,T)}$ be a smooth, connected  solution to Ricci flow and $(X,d_X)$ a  metric space.  
We say $(M,g(t))_{t\in (0,T)}$    {\rm has  initial conditions  $(X,d_X)$ }  or  {\rm is coming out of $(X,d_X)$}  if the following holds:\\
 There exists a metric $d_0$ on $M$ such that $(M,d_0)$ is isometric to $(X,d_X)$ for which 
$(M,d(g(t)) \to (M,d_0)$ {\it locally, uniformly } as $t\downto 0,$  where  $d(g(t))$ is the metric (distance)  induced by the Riemannian metric $g(t)$  for $t>0$. \\
$(M,d(g(t)) \to (M,d_0)$  {\it locally , uniformly } as $t\downto 0$   means: For all compact $K \subseteq M$  $\sup_{x,y \in K} |d(g(t))(x,y) -d_0(x,y)| \to 0$ as $t\downto 0$. 
\end{defi}

In some circumstances, a number of {\it   curvature bounds} from below  at time zero  are, up to  a constant factor, preserved for  some short time by the Ricci flow even in the case that the initial lower bound is less than zero.
 Furthermore, other estimates can then be shown to hold. 
This has lead to many applications in  the   non-smooth/smooth  setting.

To use this method for a given not-necessarily smooth metric space, one needs to answer two questions:\\
i) Does there exist a Ricci flow or related flow  of $(X,d_X)$?\\
ii) Are the curvature bounds from below of $(X,d_X)$ preserved by the flow, or  at least preserved up to a constant factor?\\

In the following sections we examine a number of 
settings where one is able to 
answer both or one of the above questions, or related questions,   affirmatively   and we explain some applications.

In  Section \ref{open-problems}  we present a list of some open problems on  Ricci flowing non-smooth objects and the preservation of lower curvature bounds. 

\section{Ricci flows of non-smooth data}\label{flows-and-curvature}

We use the following notation in this section. 
 
\begin{defi}\label{def-fair-controlled}
\  \   \\
\noindent i) For two Riemannian  metrics $g$ and $h$  on a smooth manifold $M$  and  a  constant $0<\si<\infty $ we say     {\rm $h$ is $\si$-fair  to  $g$}  if   $\frac{1}
{\si} h  \leq g \leq \si  h.$  Note that $h$ is $\si$  fair   to $g$ $\iff$   $g$ is $\si$ fair   to $h$. \\
ii) A smooth   Riemannian manifold  $(M,h)$ is said to be  {\rm geometrically bounded }    if  \! \!  
\break  $ \sup_M |\Rm(h)| < \infty$ and   $inj(M,h)>0$.\\
iii)
We use the following notation from \cite{Bur-Gro-Per}.
Let  $(X_i,d_i), (X,d_X)$ be metric spaces, $i\in \N,$ and $x_i \in X_i,$ $x\in X$ for all $i\in \N$. 
We say $(X_i,d_i,x_i)$ converges in the pointed Gromov-Hausdorff sense to $(X,d_X,x)$ if
there exists  $\ep(i),R(i) \in \R^+$ with $\ep(i) \to 0$ and $ R(i) \to \infty$ as $i\to \infty$ and maps $F_i: B_{d_i}(x_i,R(i)) \to B_{d_X}(x,R(i))$ with $F_i(x_i) =x$ 
such that $$ |d_X(F_i(y),F_i(z)) - d_i (y,z)|\leq \ep(i)$$
for all $i\in \N$ for all $x, y \in B_{d_i}(x_i,R(i)),$
and $B_{d_X}(x,R(i)- \ep(i)) \subseteq F_i(B_{d_i}(x_i,R(i)))$.  
\end{defi}
We now present an introductory, non-exhaustive,  list of   metric spaces which can be evolved by   Ricci flow,  or a related flow.  \\

{\bf Some spaces which can be evolved by Ricci flow, or a related flow}  
\begin{itemize}
\item[(i)] $g(0)\in C^0(M),$ $M$ closed.    Then there is a Ricci flow solution coming out of $(M,d(g(0))),$  Simon  \cite{SimonC01}  and  \cite{SimonC02} (see also Koch/Lamm \cite{KochLamm1, KochLamm2}).   See  Section  \ref{c0flowsection}.
\item[(ii)]  $g(0)\in C^0(M),$ $(M,h)$ smooth, complete , connected,  $\sup_M |\Rm(h)| < \infty$ ,   and assume $h$ is $ (1+\ep(n))$-fair  to $g(0)$ for a small  $\ep(n)>0$. Then there is a Ricci flow solution coming out of $(M,d(g(0))).$ Simon  \cite{SimonC01} and  \cite{SimonC02} (see also  Koch/Lamm \cite{KochLamm1, KochLamm2}).  See   Section \ref{c0flowsection}.    
\item[(iii)]Assume $(M^4,h)$ smooth, complete , connected,  geometrically bounded, four dimensional  and $g(0)\in W^{2,2}_{loc}(M^4,h)\cap W^{2,2}_{glob}(M^4,h)$  and $g(0),$
is  $a$-fair to $h$  for some $0<a< \infty,$   where 
$g(0)\in W^{2,2}_{glob}(M^4,h)$ means  
\begin{eqnarray} 
\int_M (|\gradh  g(0)|^2 +  |\gradh^2 g(0)|^2) d\vol_h  < \infty 
\end{eqnarray} 
(the $glob$  appearing in the  $W^{2,2}_{glob}(M^4,h)$ is used to note that we are referring to a global norm). 
 Then there is a smooth solution to {\it $h$-Ricci DeTurck flow} $(M,g(t))_{t\in (0,T)}$ such that $g(t) \to g(0)$ in the $W^{2,2}_{loc}$ sense as $t\downto 0$  {\bf and } a 
   Ricci flow  $(M,\ell(t))_{t\in (0,T)}$ solution coming out of $(M,d(g(0))),$ if  {\bf  $d(g(0))$ is defined appropriately} : Lamm/Simon \cite{LammSimon} (see Section \ref{weakmetrics}).  
 The Ricci DeTurck flow is introduced in the next section. 
  For the moment, we merely note that  
  $(M,\ell(t))_{t\in (0,T)}$  is a {\it Ricci flow related solution }  to $g(t)$: that is   $\ell(t):= (\Phi_t)^{*}g(t),$   
  where  $\Phi_t = \Phi(\cdot,t):M \to M$ is a diffeomorphism 
$\Phi:M \times (0,T) \to M$  is smooth. In particular  $g(t)$ is isometric to $\ell(t)$ for all $t>0$.

\item[(iv)]   
$(M^n,h)$ is smooth, complete, connected $n$-dimensional, $\sup_M |\Rm(h)| < \infty$  and $g(0)\in W^{1,n}_{glob}(M,h) \cap W^{1,n}_{loc}(M,h) $  and $g(0),$
is $a$-fair  to $h$  for some $0<a< \infty$,
where 
$g(0)\in W^{1,n}_{glob}(M,h)$ if  
\begin{eqnarray} 
\int_M |\gradh  g(0)|^n  d\vol_h  < \infty .
\end{eqnarray} 
  Then there is a smooth solution to {\it $h$-Ricci DeTurk flow} $(M,g(t))_{t\in (0,T)}$ such that $g(t) \to g(0)$ in the $W^{1,n}_{loc}$ sense. Chu/Lee, 
\cite{ChuLee}.  
  See the next section for a definition of  the Ricci DeTurck flow.  
 
\item[(v)] $(X,d_X,x)$ is the pointed Gromov-Hausdorff limit   of a sequence of smooth, three dimensional  Riemannian manifolds
$(M^3_i,g_i(0) ,x_i)$   which are uniformly non-collapsed,  
that is $\vol(B_{g_{i(0)} }(x,1)) \geq v_0>0$ for some $v_0>0$ for all $x\in M_i$, and the metrics satisfy $\Ricci(g_i(0)) + k g_i(0) \geq 0  $ for some  $k\in \R$   or
$\sec(g_i(0)) + k\geq 0$ for some  $k\in \R$. 
Then there is a smooth Ricci flow $(M,g(t))_{t\in (0,T)}$ coming out of  $(X,d_X).$
Simon   \cite{SimonCrelle3D} for the case $\sup_M |\Rm(g_i(0) )|< \infty$  for all  $i \in \N$ , Simon/Topping \cite{SiTo1, SiTo2}, Hochard \cite{hochard} for the case  $\sup_M |\Rm(g_i(0) )|= \infty$  for all or some $i \in \N.$  
\item[(vi)] $(X,d_X,x)$ is the pointed Gromov-Hausdorff limit of a sequence of smooth, connected, complete  Riemannian manifolds
$(M^n_i,g_i(0),x_i))$   which are uniformly non-collapsed,  
that is $\vol(B_{g_i}(x,1)) \geq v_0>0$ for some $v_0>0$ for all $x\in M_i$,
 and the metrics satisfy $\curlR(g_i(0)) + k I(g_i(0)) \in \curlC $ for some  $k\in \R$ where $\curlC$ is one of  the  curvature conditions  $\curlC_{CO},\curlC_{2C0}, \curlC_{IC1},
 \curlC_{IC2},$     in the list of {\it curvature conditions}  \ref{curvature-conditions} given below. 
Then there is a smooth Ricci flow $(M,g(t))_{t\in (0,T)}$ coming out of  $(X,d_X):$ 
Bamler/Cabezas-Rivas/Wilking, \cite{Bam-Cab-Riv-Wil} for  the case  $\sup_M |\Rm(g_i(0) )|< \infty$  for all   $i \in \N,$   and Bamler/Cabezas-Rivas/Wilking, \cite{Bam-Cab-Riv-Wil} Hochard \cite{HochardThesis}, Lai \cite{Lai1}, in the case    $\sup_M |\Rm(g_i(0) )|= \infty$  for some or all   $i \in \N$  : See Section \ref{sec-ricci-flow-non-collapsed-non-neg} for more details/references.
\item[(vii)] $(X,d_X,x)$ is the pointed Gromov-Hausdorff limit of a sequence of smooth, connected complete  K\"ahler manifolds
$(M^n_i,g_i(0),x_i)$   which are uniformly non-collapsed,  
that is $\vol(B_{g_i(0)}(x,1)) \geq v_0>0$ for some $v_0>0$ for all $x\in M_i$, and the metrics satisfy $\curlR(g_i(0)) + k I(g_i(0)) \in \curlC $ for some  $k\in \R$ where $\curlC$ is the condition $\curlC_{HB}$  of \ref{curvature-conditions}. 
Then there is a smooth K\"ahler Ricci flow $(M,g(t))_{t\in (0,T)}$ coming out of  $(X,d_X).$
Bamler/Cabezas-Rivas/Wilking \cite{Bam-Cab-Riv-Wil}, in the case $\sup_M |\Rm(g_i(0) )|< \infty$  for all   $i \in \N,$   Lee/Tam \cite{LeeTam} for  the case  $\sup_M |\Rm(g_i(0) )|= \infty$  for some or all   $i \in \N.$ 

\item[(viii)] $(M,g(0))$ is a smooth,  closed, K\"ahler manifold, and the potential is bounded in $C^{1,1}$ : 
Then there is a solution to K\"ahler Ricci flow with initial value given by $g(0)$ :  Chen/Tian/Zhang.  See \cite{ChenTianZhang} for details (and see Note 3 below).
\item[(ix)] $M^2$ is a two dimensional,  connected Riemann surface, and $\mu_0 $ is any nonnegative nontrivial Radon measure on $M$ that is non-atomic, that is $\mu_0(\{x\}) = 0$ for all $x \in M.$
Then there exists a smooth Ricci flow $(M,g(t))_{t\in (0,T]}$ such  
that the Riemannian volume measure, $\mu_{g(t)},$ satisfies $\mu_{g(t)} \to \mu_{0}$ weakly as $t\downto 0$ : Topping/Yin  \cite{TopYin1, TopYin2}.
\item[(x)] $(X,d_X)$ is a two dimensional closed Alexandrov space with curvature bounded from below  :  \cite{SimonCrelle3D} combined with the fact, from the theory of Alexandrov-Spaces, that $(X,d_X)$ can be written as a Gromov-Hausdorff limit of smooth Riemannian manifolds with $\sec\geq-k$ for a fixed $k\in \R$ (see \cite{Richard-T-Alex}, and  Section \ref{sec-ricci-flow-non-collapsed-non-neg} for more details and references). 
\item[(xi)]   $(M^2,g(0))$ is a smooth, connected, not necessarily complete Riemannian surface.  Giesen/Topping: See   \cite{GiTop} and the references therein.

\end{itemize}

\begin{itemize}
\item[(Note 1):] 
 Note that if $M$ is compact in (iii) or (iv),  $g \in W^{k,l}_{glob}(M^n,h)$ if and only if $g \in W^{k,l}_{loc}(M^n)$
  where $g \in W^{k,l}_{loc}(M^n)$ if and only  we can find a covering of $M$ by coordinate charts $(U_{\alpha})_{\alpha=1}^N,$ $\phi_{\alpha}: U_{\alpha} \to V_{\alpha}  \subseteq \R^n $  with inverse $\psi_{\alpha}$  for  which, 
  $ (g_{\alpha})_{ij}(x):= g(\psi_{\alpha}(x))( D\psi_{\alpha}(x)(e_i(x)),D\psi_{\alpha}(x)(e_j(x)))$ for $x \in V_{\alpha}$  satisfies   $(g_{\alpha}) _{ij} \in W^{k,l}(V_{\alpha})$ for each $\alpha \in \{1, \ldots, N\}$   for each $i,j\in \{1, \ldots,n \},$ $(k,l) = (1,n)$ in the case   (iv) above,  or $(k,l) = (2,2),$ in the case   (iii) above,   and
   $W^{k,l}(V_{\alpha})$ refers to the usual definition of a Sobolev space. 
  
   \item[(Note 2):] In the  two dimensional case \\
   a) all curvatures are controlled by the scalar curvature and\\
   b)   There is    extra structure, as one can find solutions to Ricci flow which are defined   in terms of a conformal change of the metric: $g(t)= u(t)g(0)$.
   See \cite{ham2d}. 
   We only survey (generally) two dimensional solutions in the general setting of this paper, that is in the setting of spaces with curvature bounded from below, and  we do not survey them  specifically in depth.  There are however many papers on the subject, and we refer the reader to  \cite{GiTop}, \cite{TopYin1},\cite{TopYin2}, \cite{ham2d}, \cite{HamYauHar}, and the references therein for more details.  
   \item[(Note 3)] 
  In the K\"ahler case, like the two dimensional Riemannian case, there is more structure, as one can find solutions in the closed case to K\"ahler Ricci flow which are defined   in terms of a  potential. For a given K\"ahler  form $\omega_0,$  the K\"ahler Ricci flow can be written as   $\omega(t) = \hat \omega_t  +\partial \bar{\partial} u,$  where  the potential $u : M \times [0,T) \to \R$ is a smooth scalar function, $\omega(t)$ is the K\"ahler form associated to the K\"ahler metric $g(t)$ and $\hat \omega_t = \omega_0 + t\chi,$ 
  where $\chi \in -c_1(M)$ is fixed.  The potential satisfies a Monge-Amp{\`e}re type 
  equation (see for example \cite{SonWei} for details). 
   Although  we survey some results in the K\"ahler  setting, we do so more as a subset embedded in the general setting, that is of spaces with curvature bounded from below. 
  We  do not survey them in depth in this paper. There is however a vast amount of literature on the subject, and we refer the reader to the articles \cite{SonWei},\cite{Tos} or the book \cite{BouEysGued}, and the papers referenced therein, for more details.

\end{itemize}

{\bf Curvature operators and curvature conditions}\\
  
We introduce now very briefly some notation for, and   definitions of,  curvature    conditions which  appear in this paper.
\begin{defi}
A curvature operator on $M^n $ is a $(0 \   4)$-tensor   $S,$ locally $S = S_{ijkl},$ defined on $M$ which has the same zero-order symmetries as the 
curvature tensor. 
That is, for all $i,j,k,l \in \{1, \ldots, n\}$ it satisfies: \\
$S_{ijkl} = S_{klij},$ \\
$S_{ijkl} = -S_{jikl} = -S_{ijlk},$\\
$S_{ijkl} + S_{jkil} + S_{kijl} = 0,$
\end{defi}
We use the exposition given in \cite{Bam-Cab-Riv-Wil} and \cite{Boehm-Wilking}.
In the following  $(M^n,g)$ is  a smooth Riemannian manifold and $S$ a curvature operator.
One can naturally define a metric on $\Lambda^2(T_p M),$ by  
$\langle v\wedge w , z\wedge b \rangle := g(p)(v,z)g(p)(w,b) - g(p)(v,b) g(p)(w,z)$ for all $    v\wedge w = v^i \frac{\partial}{\partial x^i}(p)  \wedge w^j \frac{\partial}{\partial x^j}(p) , 
z\wedge b = z^i  \frac{\partial}{\partial x^i}(p)  \wedge b^j \frac{\partial}{\partial x^j}(p)  \in \Lambda^2(T_p M).$ 
Then, if $e_1, \ldots,e_n\in T_p M$ is an orthonormal frame for $(T_p M,g(p))$ so is $ (e_i \wedge e_j)_{i<j, i,j  \in \{1, \ldots n\}} $   an orthonormal frame for $\Lambda^2(T_p M).$ 
One can naturally consider $S$ as  a symmetric,  bilinear      form  $S:   \Lambda^2(T M)  \times \Lambda^2(T M) \to \R$  by defining  
$S(p)(e_i\wedge e_j,e_k \wedge e_l)= S_{ijkl}(p)$ for $i<j$,$k<l,$ and  then extending this linearly.    $S$ can then    be considered as a self-adjoint (with respect to $\langle \cdot, \cdot   \rangle$ ) operator 
$S: \Lambda^2(T M) \to \Lambda^2(T M),$ $\langle S(v\wedge w)  , z\wedge b  \rangle =  \langle  (v\wedge w)  ,S( z\wedge b ) \rangle.$
Hence $S$ is diagonalisable with respect to $\langle \cdot,\cdot \rangle,$ and one can find an orthonormal basis, with respect to $\langle \cdot,\cdot\rangle$,  of eigenvectors. \\
For a smooth Riemannian manifold $(M,g)$, the Riemannian curvature tensor $\curlR = \curlR(g)$ is a curvature tensor, as is $I=I(g)$, the identity tensor given by 
$I(v,w,z,b) = \langle v\wedge w, z\wedge b\rangle.$
\begin{defi}
(Curvature conditions) \label{curvature-conditions}\\
Let $(M,g)$ be a smooth Riemannian manifold, and $S$ be a curvature operator on $M$.
\begin{itemize} 
 \item[(i)] { \rm Non-negative (positive)     Ricci curvature}\\ 
\noindent $S \in \curlC_{Ricci}$  ($S \in \curlC^{+}_{Ricci}$)
if
$ S_{kl}:= g^{ij} S_{ikjl} $ is a non-negative (positive) $(0 \ \ 2)$ tensor.\\
\item[(ii)]{\rm  Non-negative (positive) sectional curvature}\\
$S \in \curlC_{Sec}$  ($S \in \curlC^+_{Sec}$)
if
$  S(v,w,v,w)\geq 0  $  for all $v,w \in T_p M,$ ($  S(v,w,v,w)> 0  $  for all linearly independent non-zero vectors $v,w \in T_p M,$) for all $p\in M$ \\
\item[(iii)]{\rm Non-negative (positive)  curvature Operator}\\
  $S \in \curlC_{CO}$  ($S \in \curlC^+_{CO}$) 
if all eigenvalues of $S$ are non-negative (respectively positive) everywhere.\\
\item[(iv)] {\rm $2$-Non-negative (positive)  curvature Operator}\\
 $S \in \curlC_{2CO}$   ($S \in \curlC^+_{2CO}$)   if  
the sum of the two smallest eigenvalues of $S$ are non-negative  respectively  the sum
of  the two smallest eigenvalues of $S$ are positive everywhere..\\
\item[(v)] {\rm   WPIC :Weakly positive isotropic curvature (respectively PIC:  positive isotropic curvature)   }\\
   $S \in \curlC_{IC}$  (  $  S \in \curlC^+_{IC}$ ) if
$$ S_{1313} +  S_{1414} + S_{2323} +  S_{2424} - 2 S_{1234} \geq 0,$$ (respectively 
$>0$) for all orthonormal frames $\{e_1,e_2,e_3,e_4\}$
in $T_p M$ for all $p\in M.$\\ 
\item[(vi)]{\rm WPIC1 and PIC1 }\\
$S \in \curlC_{IC1}$  ( $  S \in \curlC^+_{IC1}$  )  if
$$ S_{1313} + \lambda^2 S_{1414} + S_{2323} +\la^2  S_{2424} - 2 \la S_{1234} \geq 0,$$ (respectively 
$>0$) for all orthonormal frames $\{e_1,e_2,e_3,e_4\},$
in $T_p M$ for all $p\in M,$ and all  $\la \in [0,1].$\\
\item[(vii)]{\rm  WPIC2 and PIC2  }\\
$S \in \curlC_{IC2}$ ( $ S \in \curlC^+_{IC2}$ ) if
\\
$$S_{1313} + \lambda^2 S_{1414} +\mu^2S_{2323} +\la^2 \mu^2 S_{2424} - 2\mu \la S_{1234} \geq 0 $$  (respectively 
$>0$) for all orthonormal frames $\{e_1,e_2,e_3,e_4\},$
in $T_p M$ for all $p\in M,$ and all  $\mu,\la \in [0,1].$
\item[(viii)] {\rm  Non-negative (positive)  holomorphic bisectional curvature}\\
If  $(M,g)$ is K\"ahler, with complex structure $J$,  then we say $S \in \curlC_{HB} $  \\
if $R(p)(X,JX,Y,JY)\geq  0$ for  all $X,Y  \in T_p M$ ($S \in \curlC^+_{HB} $  
if $R(p)(X,JX,Y,JY) >  0$ for all non-zero $X,Y  \in T_p M$) for all $p\in M$.

\end{itemize}
\end{defi}

Note: The conditions $S \in  \curlC^+_{IC},$ respectively $S\in  \curlC^+_{IC1}, $ respectively $S \in \curlC^+_{IC2}, $ are generally  known  in the literature  as  "$S$ has positive isotropic curvature ", or $S$ is PIC, respectively "$S$ is PIC1", respectively
"$S$ is PIC2". The condition PIC was introduced in the paper \cite{MicMoo}, 
the conditions PIC1, PIC2    in the paper  \cite{Brendle-Schoen}. 

Also :  $S \in \curlC$ {\bf except case (v)} implies $S\in \curlC_{Ricci}$.
For a list of some more relations between the curvature conditions above,  we refer the reader to   the book of Brendle  \cite{BrendBook} or the survey of Topping \cite{Top-Sur}.

This is a very brief introduction  to the curvature    conditions we shall consider. 
For the very many alternative, often more geometric,  but equivalent  definitions of theses conditions, we refer the reader to  \cite{HamFor} ,\cite{ham4d},  \cite{HamIso} ,  \cite{MicMoo}   (where the PIC condition was introduced ),   \cite{Brendle-Schoen} ( where the PIC1 and PIC2 condition were introduced),    as well as  \cite{WilkingLie} (where    Lie-algebraic descriptions of many  curvature conditions was   studied in the context of the Ricci flow),   \cite{Bam-Cab-Riv-Wil} (see Section \ref{sec-ricci-flow-non-collapsed-non-neg} and   \cite{Bam-Cab-Riv-Wil}  for more details), \cite{Top-Sur} for a detailed survey of the PIC1 condition in the Ricci flow context,   and \cite{HuyNguyenThesis}.

 


\section{$C^0$  Limits of smooth Riemannian manifolds   with      curvature lower bounds}\label{c0flowsection}

In the papers \cite{SimonC01} and  \cite{SimonC02}   (in  \cite{SimonC02}  some corrrections to \cite{SimonC01} were made) it was shown for   any $C^0$ Riemannian metric $g(0)$ on a closed manifold $M$, and any  smooth metric $h$, which  satisfies $ \frac{1}{(1+\ep(n))}h \leq g(0) \leq (1+\ep(n))h $ for a small constant $\ep(n)>0$ (such an $h$ may be obtained by mollifying the metric $g(0)$), then there exists a solution $g(t)_{t\in (0,T)}$
to the so called {\it Ricci DeTurck flow with background metric $h$}, such that $|g(t)- g(0)|_{h} \to 0$    as 
 $t \downto 0,$ where $T = T(n,\sup_M |\Riem(h)|_h)  >0.$
 The Ricci DeTurck flow was first studied by DeTurck in the paper \cite{DeTurck}.
 For simplicity  we call  {\it Ricci DeTurck flow with background metric $h$} the { $h$-Ricci DeTurck flow}.
 For an interval $I$, $(M,g(t))_{t \in I}$ is a solution to $h$-Ricci DeTurck flow if it solves the equation
 \begin{eqnarray} \label{Meq}
  \partt {g}_{ij}=&\,{g}^{ab} (\gradh_a\gradh_b {g}_{ij})
  -{g}^{kl}{g}_{ip}h^{pq}R_{jkql}(h)
  -{g}^{kl}{g}_{jp}h^{pq}R_{ikql}(h)\cr
  &\,+\tfrac12{g}^{ab}{g}^{pq}\left(\gradh_i {g}_{pa}\gradh_j{g}_{qb}
    +2\gradh_a{g}_{jp}\gradh_q{g}_{ib}-2\gradh_a{g}_{jp}
    \gradh_b{g}_{iq}\right.\cr
  &\,\left.\qquad\qquad\qquad-2\gradh_j{g}_{pa}\gradh_b{g}_{iq}
    -2\gradh_i{g}_{pa}\gradh_b{g}_{jq}\right),
\end{eqnarray}
 in the smooth sense on $M\times I$, where here, and in the rest of the paper,  $\gradh$ refers to the covariant derivative with respect to $h$. 
 Ricci DeTurck flow and  Ricci flow in the smooth  setting  are closely related : 
 given an  $h$-Ricci DeTurck flow $g(t)_{t\in I}$ on a closed  manifold and an $S \in I$ there  is a smooth  family of diffeomorphisms $\Phi(t):M \to M,$ $t\in I$ with $\Phi(S) = Id$ such that $\ell(t) = (\Phi(t))^*g(t)$  is a smooth solution to Ricci flow. The diffeomorphisms $\Phi(t)$ solve the following ordinary differential equation:
 \begin{align}
    \partt {\Phi}^{\al}(x,t) =& V^{\al}(\Phi(x,t),t),  \ \ \ \text{for all}\ \  
    (x,t)\in M^n \times I,\nonumber \\
     \Phi(x,S)=& x.  \label{ODEDe}
  \end{align}
  where $ {V}^{\al}(y,t) := -{g}^{\be \ga} \left(
      {\Gamma(g)}^{\al}_{\be \ga} -  {\Gamma(h)}^{\al}_{\be \ga}\right)
  (y,t)$

\begin{thm}(  \cite{SimonC01} and \cite{SimonC02} )\label{C0RicciDeturckthm}
For any $n\in \N$ there exists an $\ep(n)>0$ such that the following holds. 
Let   $(M^n,h)$ be a smooth, connected, complete $n$-dimensional 
Riemannian manifold such that $K_j:= \sup_M |\grad^j \Rm(h)|  <  \infty$  for all  $j \in \N_0.$
Assume further that    $g(0)$  is  a $C^0$  Riemannian metric on $M,$  for which   $\frac{1}{(1+ \hat \ep )}h \leq g(0) \leq (1+\hat \ep )h$  holds, where $0<\hat \ep \leq \ep(n).$ Then  
 there exists  a   $T=T(n,\hat \ep, K_0) >0$, $T$ being monotone increasing in the second variable and decreasing in the third variable, and  a  smooth complete  solution $(M,g(t))_{t \in (0,T)}$  to \eqref{Meq}
with initial value $g(0)$, in the sense that   $\sup_{K} |g(t)-g(0)|_{g(0)}\to 0$ for any compact $K \subseteq M,$ as $t\downto 0,$ and    
the following estimates hold:
\begin{eqnarray*}
&& \frac{1}{(1+ 2\hat \ep )}h \leq g(t) \leq (1+2\hat \ep )h \\
&&|\gradh^j g(t)|^2 \leq \frac{c_j(n,h)}{t^j}   
\end{eqnarray*}
for all $t\in (0,T).$
If $g(0)$ is smooth, then the solution may be smoothly extended to $(M,g(t))_{t \in [0,T)}.$ 
Here $c_j(n,h) =a(j,n,\sup_M |\Riem(h)|,  
\ldots, \sup_M |\grad^j \Riem(h)|)$ and $a(j,n,k_0,k_1, \ldots,k_j)$ is monotone increasing in each of the  
variables $k_0,k_1, \ldots,k_j$ for all $j \in \N_0,n \in \N$. 
\end{thm}

\begin{rem}
If $(M,\hat h)$  is a smooth, complete, connected  manifold with $ \hat K_0 := \sup_{M} |\Rm(\hat h)|< \infty,$ and $g(0)$ a $C^0$ Riemannian metric  such that  $g(0)$ is $1 + \frac{\ep(n)}{2}$ fair to $\hat h$, then we can find, using Theorem 1.2 of Shi (\cite{shi}), 
an  $h$ satisfying the conditions of the Theorem. 
\end{rem}
\begin{rem}
If $M$ is closed, smooth, and $g(0)$ a given $C^0$ Riemannian metric, then   an $h$ and $\hat \ep$  as in the statement of the theorem always exist.
\end{rem} 

In later papers, Koch and Lamm \cite{KochLamm1}, \cite{KochLamm2}   proved   versions of   Theorem  \ref{C0RicciDeturckthm}  and  short time existence for many other non-linear flows with irregular initial data, using fixed point arguments and arguments from harmonic-analysis (see \cite{KochLamm1,KochLamm2}).

In a further paper   Lamm/Simon \cite{LammSimon} it was shown, as a corollary  to the proofs and results of that paper,  that there is a Ricci flow related solution coming out of $(M,d(g(0)))$, if one further assumes
there exists an $i_0>0$ such that $inj(M,h)\geq i_0>0$ : See Section \ref{weakmetrics},  Theorem \ref{WeakRicciStart}.  
\begin{thm}(Lamm/Simon \cite{LammSimon})\label{C0Ricciflow-Thm}
There exists an  $\ep_1(n)>0 $ with the following properties.  
Let $M^n$ be smooth,    $g(0),h,$ $\hat \ep, \ep(n)$  be as in 
 in the statement of Theorem   \ref{C0RicciDeturckthm}  and assume further that $inj(M,h)\geq i_0>0,$ and $\ep(n) < \ep_1(n)$. 
Then there exists a smooth, complete, Ricci flow solution $\ell(t)_{t \in (0,T)}$   with    $\sup_{K} |d(\ell(t))-  d_0| \to 0$ for any compact $K \subseteq M,$ as $t\downto 0,$   where $d_0$ is a metric isometric to $d(g(0))$.  The solution $\ell(t)$ is a Ricci flow related solution to the Ricci DeTurck solution $g(t)$ from Theorem \ref{C0RicciDeturckthm}, and can be written as $\ell(t):= (\Phi_t)^{*} g(t),$ where $\Phi:M \times(0,T) \to M$ is a solution to \eqref{ODEDe} with $S=\frac{T}{2},$ and  $\Phi_t := \Phi(\cdot,t)  $ are diffeomorphisms for $t\in(0,T),$  $\Phi_t \to \Phi_0$ locally uniformly as $t\downto 0$ where $\Phi_0$ is a   homeomorphism. The metric $d_0$ is given by $d_0:= \Phi_0^{*}(d(g(0)))$ 
\end{thm}
\begin{proof} 
We let $\Phi:M \times (0,T) \to M$ be the solution to 
\eqref{ODEDe2} with $S=\frac{T}{2}$. 
The paper of Lamm/Simon then shows (see Theorem 2.6 in \cite{LammSimon}) that $\Phi(\cdot,t) \to \Phi_0$ locally uniformly as $t\downto 0$ where $\Phi_0$ is a   homeomorphism. Now using that $d(g(t)) \to d_{g(0)}$ locally uniformly, we see that  $d_{\ell(t)}(x,y) = 
 d(g(t))(\Phi_t(x), \Phi_t(y))$ satisfies $d_{\ell(t)} \to 
  (\Phi_0)^{*}d_{g(0)} =:d_0$ locally uniformly. The completeness may be seen as followed: 
  For fixed $t$ let $(x_i)_{i \in \N}$ be a Cauchy sequence  with respect to $d(\ell(t))$.
  We have,     $d_{\ell(t)}(x_i,x_j) \leq \ep(N)$  for all $i,j \geq N$ where
  $ \ep(N)\to 0$ as $N \to \infty.$ Hence 
  $d_{g(t)}(\Phi_t(x_i),\Phi_t(x_j) )  \leq \ep(N)$ for all $i,j \geq N$. Since $g(t)$ is $a$-fair to $h$, we know that $g(t) $ is complete. 
  Hence, $\Phi_t(x_i)$ converges with respect to $g(t)$ to an $p$ as $i\to \infty$ which we can write uniquely as $p= \Phi_t(q),$ since $\Phi_t$ is a smooth diffeomorphism.  
 Hence $(x_i)_{i\in \N}$ converges, as $i\to \infty$, with respect to $\ell(t)$ to $q$.
 \end{proof}

We note that Theorem \ref{C0RicciDeturckthm} and the proof thereof    in \cite{SimonC01}  implies the following
\begin{thm}\label{c0_curvature_control}
Let $(M,g(0))$ be $C^0$ and $h$ be fixed, and satisfy 
$\sup_M |\grad^j \Rm(h)|_h=:K_j < \infty.$ Let $\hat \ep \leq \ep(n)$ 
satisfy $\frac{1}{(1+ \hat \ep )}h \leq g(0) \leq (1+\hat \ep )h$   and let 
$ g(t)_{t\in (0,T)}$ be the solution to \eqref{Meq} from
Theorem \ref{C0RicciDeturckthm}. 
Then there exists $ \hat T(n,h,\hat \ep) >0, \hat T\leq T$ such that
$$|\Rm(  g(t))|\leq \frac{\psi(n,\hat \ep)}{t}$$ for all $t\in (0,\hat T(n,h,\hat \ep)),$ where $\psi(n,\hat \ep)\to 0$ as $\hat \ep \to 0$. 
Hence the same is true for any Ricci flow related solution. In particular, if 
 $inj(M,h)>0,$ then  
  the Ricci flow related solution obtained in Theorem \ref{C0Ricciflow-Thm} satisfies this estimate  also: $$|\Rm(  \ell(t))|\leq \frac{\psi(n,\hat \ep)}{t}$$
for all $t\in (0,\hat T(n,h,\hat \ep)),$ where $\psi(n,\hat \ep)\to 0$ as $\hat \ep \to 0$. 
\end{thm}
\begin{proof}
Let $t_0$ be a  time with $t_0 \leq \hat T(n,h, \hat \ep)$ small, $\hat T(n,h, \hat \ep)$ to be specified below. 

We scale the solution, by $\hat g( \hat t) = cg( \frac{\hat  t}{c})$ and $\hat h =ch$ with $c= \frac{1}{t_0}$. If $t_0  \leq \hat T(n,h, \hat \ep)$ is chosen small enough,
then we have 
$\sum_{j=0}^{4} |\gradhath^j \Rm(\hat h)|^2 \leq \hat \ep.$ Assume  for the moment that  $inj_h(x) \geq 1$ for all $x \in B_{\hat h}(p,1).$ 
Then  there exist  geodesic coordinates on a ball of radius one centred at $p$ such that $|\hat h-\de| + |D\hat h| + |D^2 \hat h|\leq \hat \ep$ in these coordinates.
The proof of the estimates in   \cite{SimonC01} and  \cite{shi}  is still valid for $\hat g$ and we obtain
$\hat t |\gradhath \hat g|^2(\cdot,\hat t)   + 
\hat t^2 |\gradhath^2 \hat g|^2(\cdot,\hat t)| \leq 1$ for all $\hat t\leq 1$.
Using these estimates for $\hat t=1$, with  the estimates for $\hat h$ in  geodesic coordinates and  the fact that $ \frac{1}{(1+ 2\hat \ep )}\hat h \leq \hat g(t)  \leq (1+2\hat \ep )\hat h$ for all $\hat t\leq 1$ (a scale invariant fact),
we see, by interpolating,  that $|\Rm (\hat g)(p,1)| \leq \psi(\hat \ep)$. This implies
$|\Rm (g)(p,t_0)| \leq \frac{\psi(\hat \ep)}{t_0}$ as required.\\
If the condition on the injectivity radius near $p$ is not satisfied, then we lift $\hat g$ and $\hat h$ to a ball of radius one using the exponential map, and repeat the argument there to obtain the same conclusion.
\end{proof}

Using the $h$-Ricci DeTurck flow in this $C^0$ setting,  and its relation to the Ricci flow,  it is possible to show that certain curvature bounds from below will, up to a constant, be preserved.  
In the next section, Section \ref{sec-ricci-flow-non-collapsed-non-neg}, we  will see that some similar results also  hold in a more general setting, for some, but not all,  of the curvature conditions considered here.

We gather a number of  results in the theorems below.
Burkhardt-Guim studied the setting of scalar curvature bounded below in the setting of $C^0$ metrics as stated, \cite{Burk-Guim1, Burk-Guim2}.

\begin{thm}(Burkhardt-Guim  \cite{Burk-Guim1})  \label{c0scalar}
 Let $M^n$ be closed, smooth, $g(0)$ a $C^0$   metric on $M$  and $k\in \R,$ and assume there exists a sequence of smooth 
metrics $g_i(0)$ such that $|g_i(0)-g(0)|_{C^0}\to 0$ as $i\to \infty,$  such that 
   $\Sc(g_i(0)) \geq  k -\frac 1 i.$  
Then the solution to the $h$-Flow ($h=g_i(0)$ for $i$ large and fixed) $ g(t)$ for $t\in [0,T(n,h)]$ from Theorem \ref{C0RicciDeturckthm} satisfies   $\Sc(g(t)) \geq k$ for all $t \in (0,T).$
\end{thm}

\begin{thm}( Schlichting \cite{PhDschlichting},  Richard \cite{Richard} , Simon (non compact case: proof below) ).\label{Schlichting}
 Let $M^n$ be smooth, connected,   $g(0)$ a $C^0$  Riemannian  metric on $M,$ so that $(M,g(0))$ is complete,   and assume there exists a sequence of smooth 
metrics $g_i(0)$ with $\sup_M |\Rm(g_i(0))|< \infty,$ $inj(M,g_i(0))>0,$ such that  $(M,g_{i}(0))$ is complete,   $|g_i(0)-g(0)|_{g(0)}\to 0$ as $i\to \infty,$  and  $k, k_i \in \R^+_0$    such that 
 $\curlR(g_i(0)) + k_i I   \in \curlC$  
for any of the curvature conditions $\curlC$ from    (i) -(vii) in \eqref{curvature-conditions},and $k_i \to k$ as $i\to \infty.$ 
Then there is a smooth $h$  and a solution to the $h$-Flow $ g(t)$ for $t\in [0,S(n,h,\max(k,1))]$ with inital value $g(0)$,     satisfying the estimates in and conclusions  of  Theorem \ref{C0RicciDeturckthm} and  
$\curlR(g(t)) + k C(n)I   \in \curlC$ for all $t\in  [0,S(n,h,\max(k,1))].$ 
\end{thm}
\begin{rem} \label{RemSchlichting}
There is also a Ricci flow related  solution $\ell(t)$ to $g(t),$ satisfying $\curlR(\ell(t)) + k C(n)I   \in \curlC,$ where 
 $\ell(t):= (\Phi_t)^{*} g(t),$ where $\Phi:M \times(0,T) \to M$ is a solution to \eqref{ODEDe} with $S=\frac{T}{2},$ and  $\Phi_t := \Phi(\cdot,t)  $ are diffeomorphisms for $t\in(0,T),$  $\Phi_t \to \Phi_0$ locally uniformly as $t\downto 0$ where $\Phi_0$ is a   homeomorphism, 
 $d(\ell(t)) \to  d_0$ locally uniformly, where  $d_0$ is given by $d_0:= \Phi_0^{*}(d(g(0)))$ 
\end{rem}

{\it Proof of Theorem \ref{Schlichting} and Remark \ref{RemSchlichting} } 
 By using Shi's Theorem (Theorem 1.2 of \cite{shi}) applied to one of the metrics $\hat h := g_{k}(0)$ we can find a metric $h$ with
 $\sup_M |\grad^j \Rm(h)|^2 < \infty,$  for all $j\in \N_0,$  $inj(M,h)\geq i_0>0$   and $|g_{i}(0)-h|\leq \hat \ep,$ 
 for all $i\in \N$, after taking a subsequence if necessary, where  $\hat \ep$ as small as we like.  
Hence without loss of generality,  $|g_i(0)-h|_h \leq \hat \ep  \leq \ep(n)$ for all $i\in \N$, where $\hat \ep>0$ is a constant as small as we like.
Let $  g_i(t)_{t\in [0,T(n,h)]}$ be the Ricci DeTurck $h$-flow solution 
from Theorem \ref{C0RicciDeturckthm} with initial value $g_i(0)$. 
 By  choosing $\hat \ep$ small enough, and reducing $T$ if necessary,  we have, 
 $ \sup_M |\Rm(g_i(t))| \leq  \frac{\si(n)}{t}$ for all $t\in (0,T)$
 for all $i\in \N,$ due to   Theorem  \ref{c0_curvature_control}, where $\si(n)$ is a small constant of our choosing.
We choose $S=T/2$ and  define $\Phi_i(x,S)=x,$ and let $\Phi_i: M \times [0,T] \to M$ be the solution to \eqref{Meq}. Then $\ell_i(t)= (\Phi_i)(\cdot,t)^*(g_i(t))_{t\in (0,T]}$ is a  smooth solution to Ricci flow, in view of Theorem \ref{C0Ricciflow-Thm}. 
Furthermore,  Theorem \ref{C0Ricciflow-Thm}, tells us that 
 $\Phi_i(t)\to \Phi_i(0)$ locally uniformly  as $t\downto 0,$ where $\Phi_i(0):M \to M$ is a homeomorphism. In this instance the convergence is locally smooth and $\Phi_i(0)$ is a diffeomorphism, in view of the fact that $g_i(0)$ is smooth.
 Hence $\ell_i(t) = (\Phi_i(t))^{*}g_i(t)$ is a smooth solution to
 Ricci flow with initial smooth value $\ell_i(0):= (\Phi_i)^*(g_i(0))$, and satisfies $|\Rm(\ell_i(t))| \leq \frac{\si}{t}$ and $\sup_M |\Rm(\ell_i(0))|< \infty$. Theorem 1.2 of  Lee's paper, \cite{Lee2} (see also Simon \cite{Sim08}, Chen \cite{BLChen09}),    now guarantees that
 $(M^n,\ell_i(t))_{t \in [0,T]}$ is a bounded curvature solution.
 \\ 
In the   Ph.D. Thesis of Schlichting, \cite{PhDschlichting}, 
it was shown for Ricci flows  $(M,\ell(t))_{t\in [0,T]}$ satisfying   $|\Rm(\ell(t))|\leq \frac{\si(n)}{t},$ with $\si(n)>0$ sufficiently small,   and $\sup_{M\times [0,T]} |\Rm(\ell)|< \infty $ 
and   $\curlR(\ell(0)) +  m I \in \curlC_{IC},$ (condition (v) of the curvature conditions \eqref{curvature-conditions}) that then 
$\curlR(\ell(t)) +  C(n) m I \in \curlC_{IC}$ for all $t\in [0,R(n,\max(m,1) )]\cap [0,T].$
In Richard, \cite{Richard} all the other cases were shown:  
  $\curlR(\ell(0)) +  m I \in    \curlC $  then 
$\curlR(\ell(t)) +  C(n) m I \in  \curlC$  for all $t\in [0,R(n,\max(m,1))]\cap [0,T] $ for all the remaining curvature conditions $\curlC$, (i)-(iv) and (vi), (vii)  in \eqref{curvature-conditions}. These estimates  were  proven by using the technique introduced in Simon \cite{SimonCrelle3D} (see the next section), which involves adding   a factor of $t\Sc(t)$  and a constant to the tensor being considered  and calculating the evolution equation thereof. 
 One shows that 
\begin{eqnarray*}
  \curlT(t) := \curlR(\ell(t)) + t m A(n)\Sc(\ell(t)) +   (1+Z(n)t)m I 
  \end{eqnarray*}
   remains in $\curlC$ in view of  the maximum principle  of  Hamilton, and the evolution equation satisfied by $\curlT$,   with the help of the fact that
     $|\Rm(\ell(t))|\leq \frac{\si(n)}{t}$ holds, and  $\si(n)>0$ is  sufficiently small.   

Hence in our setting  
$\curlR(\ell_i(t)) +  C(n) k_i I  \in  \curlC$  for all the curvature conditions $\curlC$  in (i)-(vii).
This implies the same for $g_i(t),$ as $g_i(t)$ is  isometric to $\ell_i(t).$ Taking a limit
in $i$ and using these estimates  and those of Theorem \ref{C0RicciDeturckthm}, gives a solution $g(t)$ to the Ricci DeTurck flow satisfying  $\curlR(g(t)) +  C(n) k I  \in \curlC,$ 
 for all $t\in [0,S(n,k)]$ as required.
 Also, taking a limit in $i$, we obtain  a solution $\Phi:M \times (0,T) \to M$ and an  $\ell(t)$ to Ricci flow, as described in the Remark: The fact that $\Phi_t:M \to M$ is a diffeomorphism for all $t\in (0,T),$ and $\Phi_t $ converges locally uniformly to a homeomorphism $\Phi_0$ follows from the proof of  Theorem 2.6 in \cite{LammSimon}.  
 $\Box$

\begin{rem}
The  proof  in Schlichting and Richard, of the fact  that $\curlT$ remains non-negative requires crucially that   $|\Rm(\ell(t))|t \leq  \si(n),$ with $\si(n)$ small enough: 
if one has a solution with  $|\Rm(\ell(t))|t \leq  c^2_0,$ for some general $c_0$, the proof that $\curlT$ remains non-negative,  possibly with  $A$ and $Z$   now also depending on $c_0$, given in Schlichting and Richard,    fails. This should  be compared    
to the special case of $n=3,$ where it was  shown in \cite{SimonCrelle3D} that for any smooth, connected, complete solution $(M,\ell(t))_{t\in [0,T]}$ to Ricci flow with \\ $\sup_{M \times [0,T]} |\Rm(\ell(t))| < \infty$ , then   $  \Ricci(\ell(t)) + (1+100t)  k (t \Sc(\ell(t))) +  k  I (1+100t) \ell(t) $  remains non-negative  for $t\in [0,1/100]\cap [0,T],$ $k\in [0,1]$  if it is non-negative at time zero,   {\bf without any further assumptions }: see   Section   \ref{sec-ricci-flow-non-collapsed-non-neg}. 
 \end{rem}

\begin{rem}
In \cite{SimonCrelle3D} for the case $n =3$
and \cite{Bam-Cab-Riv-Wil} for general $n\in \N$ the authors showed that similar    curvature lower bound estimates to those appearing in the theorem, with the exception of positive isotropic curvature, hold in  a more general setting : See the next section.

\end{rem}
 
\begin{rem}
The stronger condition  that the initial metric is  Lipschitz 
was considered in the paper \cite{SimLip}.
In that case, one obtains improved estimates.  For example in the setting of  Theorem \ref{C0RicciDeturckthm}, with the extra initial  condition $|\grad^h g(0)|\leq c_0$ one obtains 
$|\Rm(g(t))|\leq \frac{c}{\sqrt t}$ 
\end{rem}

 \begin{rem}
There are  similar results in the K\"ahler setting : See for example the paper \cite{ManChunTam}.
 
\end{rem}

\section{Non-collapsed limits of smooth spaces with curvature lower bounds}\label{sec-ricci-flow-non-collapsed-non-neg}

As explained in the introduction, assuming $M$ is smooth, closed (or that   one has a bounded curvature  solution to Ricci flow), the non-negativity of a number of  curvature conditions is   preserved by the Ricci flow.

In the author's Habilitation Thesis, \cite{SimonHabil}  and the subsequent paper \cite{SimonCrelle3D}  the author considered the problem of which negative  lower  bounds on curvature will, up to a constant, be preserved.  
The following  Lemma shows that there is some quantitative preservation, up to   constants,   of various lower bounds of curvature in three dimensions.
\begin{lem} (\cite{SimonCrelle3D})
Let $(M^3,g(t))_{t\in [0,T]}$ be a smooth, complete, connected three dimensional solution to Ricci flow which is compact or satisfies $\sup_{M \times[0,T]}|\Rm(g) |_{g} < \infty$.\\
Then for $0\leq \ep \leq \frac{1}{100}$  the following holds: 
  \begin{itemize}
 \item[(a)] 
    $  \Rc(g(0)) \geq -\ep\Sc(g(0)) g(0),$ \\ $\implies$  
    $\Rc(g(t)) \geq -\ep(1+4t) g(t)-\ep(1+4t)\Sc(g(t))  g(t)$   {\rm for all}  $t \in [0,T]\cap[0,\frac{1}{8}]$ 
  \item[(b)] 
$ \sec(g(0)) \geq  -\ep $
\\ $\implies$   $\sec(g(t)) \geq  -\ep(1+4t)  -\ep(1+4t)\Sc(g(t))$    {\rm for all} $t \in [0,T]\cap[0,\frac{1}{8}]$
  \item[(c)]  
 $\Rc(g(0)) \geq -\ep g(0) $  \\ $\implies$  
 $\Rc(g(t)) \geq -\ep(1+100t) g(t)-\ep(1+100t)t\Sc(g(t)) g(t)  $ { \rm for all } $t \in [0,T]\cap[0,\frac{1}{200}]$
 \item[(d)]
$\sec(g(0)) \geq  -\ep$ \\ $\implies$  
    $\sec(g(t)) \geq  -\ep(1+100 t)  -\ep(1+100t)t \Sc(g(t))$
  { \rm for all } $t \in [0,T]\cap[0,\frac{1}{200}]$
 \end{itemize}
\end{lem}
\noindent{\it Proof ideas}.\\ 
The proof of the Lemma involves examining the evolution of Tensors/ geometric inequalities under the Ricci flow and using the Maximum principle of \cite{ham3d, ham4d}.
For example in case (c), on a closed manifold,  one first writes  the evolution equation  of the tensor $T(\cdot,t)  =  \Rc(g(t))  +   \ep(1+100t) g(t)+   \ep(1+100t)t\Sc(g(t)) g(t) $ 
as $\partt T = \lap_g T  + L$ : with the help of the evolution equations 
 for the Ricci curvature and scalar curvature computed in \cite{ham3d}, one  can explicitly calculate $L$. Then, using the Maximum Principle of Hamilton, one shows that  $T(\cdot,0) > 0 \implies T(\cdot,t) > 0.$
 The term  in $L$ coming from $(\partt -\lap_{g(t)})(t\Sc(g(t))),$ is a non-negative  term, and plays an important role 
 in the maximum principle argument : It guarantees  that $L$ will be positive at any potential first time $t_0>0$ and point $p_0$ and non-zero direction $v_0$ where $T(p_0,t_0)(v_0,v_0) =0$ 
$\Box.$\\
We notice that if we restrict ourselves to the class of smooth, complete, connected three dimensional    solutions $(M,g(t))_{t\in (0,T)}$ as in case (c) and (d) 
of the Lemma above, but further assume that 
 \begin{eqnarray*}
&& \Sc(g(\cdot,t))\leq \frac{c^2_0}{t}  \ \text{\rm for all }  t \in [0,T],
 \end{eqnarray*}
  then (c) and (d) of the above Lemma implies the following:
  \begin{lem}
 Let $(M^3,g(t))_{t\in [0,T]}$ be a smooth, complete, connected three dimensional solution to Ricci flow which is compact or satisfies $\sup_{M \times[0,T]}|\Rm(g) |_{g} < \infty$.\\
 Assume for some $\ep\in [0,\frac 1 {100}]$ that 
    $\Rc(g(0)) \geq  -\ep g(0) $      $( \sec(g(0)) \geq -\ep )$  and that the solution satisfies  $|\Sc(g(t))| \leq \frac{c^2_0}{t}.$  
    Then   
 $\Rc(g(t)) \geq -\ep (1+ c^2_0 )(1+100t)  g(t)    $  
 $( \sec(g(t)) \geq  -\ep(1+c^2_0) (1+100 t) )$   { \rm for all } $t \in [0,T]\cap[0,\frac{1}{200}]$.  
\end{lem}

Once we are in this setting, it is known    that estimates on distance , $d_t(\cdot ,\cdot)$,   where $d_t(x,y) = d(g(t))(x,y)$ denotes the distance measured with respect to $g(t)$ between the two points $x$ and $y$,  and volume of metric balls, measured with respect to the evolving metric,  can be proven.
 We write the  version shown in  \cite{SimonCrelle3D}, although local versions also exist : see   \cite{SiTo1} for some  local versions. 
 \begin{thm}(\cite{SimonCrelle3D}) \label{distance_and_volume_estimates_ricci_flow}
 Let $(M^n,g(t))_{t\in [0,T]}$ be a smooth, connected complete solution to Ricci flow  with bounded curvature   satisfying
 \begin{eqnarray*}
 &&\sup_{x\in M, t\in[0,T)}  |\Rc(g(t))|(x)\leq \frac{c^2_0}{t}\\
 && \Rc(g(t)) \geq -k_0^2\\
 &&\vol(B_{g(0)}(x,1)) \geq v_0>0\\
 && \text{ \rm for all } \ x  \in M
 \end{eqnarray*}
 for some constants $0<v_0< \infty,1<c_0< \infty$.
 Then  
 \begin{eqnarray*}
 &&    e^{k^2_0(t-s)}d_s(x,y) \geq  d_t(x,y) \geq d_s(x,y) - \ga(n)c_0\sqrt{s-t} \\
 && \vol(B_{g(t)}(x,1)) \geq \frac{v_0}{2}>0\\
 && \text{ \rm for all } \ x,y \in M, \ s \leq t, \ t,s  \in [0,T]
 \end{eqnarray*}
 \end{thm}
  {\it Proof ideas}.
Distance estimates from below of the type $$d_t(x,y) \geq d_s(x,y) - \ga(n)\int_{s}^t \sqrt{\sup_{M} |\Rc(g(r))|} dr$$ where shown in Hamilton \cite{HamFor}, and hence, using $\sqrt{\sup_{M} |\Rc(g(r))|} \leq \frac{c_0}{\sqrt r},$ we get the lower bound on the evolving distance. The upper bounds follow from the fact 
$\partt  g(t) = -2\Rc(g(t) \leq 2k_0^2g(t),$ after integrating. 
The volume estimates now follow from a contradiction argument combined with the theorems on volume convergence of Cheeger/Colding
(\cite{Cheeger-Colding}, \cite{Cheeger_notes}). See Lemma 6.1 and Corollary 6.2 in \cite{SimonCrelle3D} for details.
$\Box.$

We remind the reader of one of  the Theorems of Perelman on ancient solutions:
\begin{thm} (\cite{Per1})\label{PerThm}
 Any    smooth, complete, connected, ancient, non-trivial bounded curvature solution  $(M^n,g(t))_{t\in (-\infty,0]}$  to Ricci flow with   non-negative curvature operator 
    must have  zero asymptotic volume ratio    that is 
\begin{eqnarray*}
  AVR(g(t)) := \lim_{r \to \infty} \frac{\vol({B_{g(t)}}(x,r))}{r^n} 
\end{eqnarray*}
exists and satisfies $  AVR(g(t)) =0$ for all $t\in (-\infty,0]$. 
\end{thm}

Combining these three theorems, we see that it is possible in three dimensions to show that uniformly non-collapsed metrics  with Ricci curvature bounded below,    and bounded curvature,  can be flown for a time $T$ which depends only on the non-collapsing constant and the  Ricci curvature bound from below, and uniform estimates will hold 
for the time interval $[0,T)$.
\begin{thm}(\cite{SimonCrelle3D})\label{simon-main}
Let $(M^3,g(0)) $ be a smooth, connected, complete Riemannian manifold with bounded curvature  and
\begin{eqnarray*}
 && \Rc(g(0)) \geq -k \\
 &&\vol(B_{g(0)}(x,1)) \geq v_0>0\\
 && \text{ \rm for all } \ x  \in M
 \end{eqnarray*}
for some $k > 0$. Then there exist  $S(v_0,\max(k,1))>0,$  $c_0 = c_0(v_0) > 0,$
a solution to Ricci flow
$(M,g(t))_{t\in [0,S]}$ such that
\begin{eqnarray*}
&& \Rc(g(t)) > -2c^2_0 k \\
&& \sup_M |\Rm(\cdot,t)|< \frac{c^2_0}{t}\\
 &&    e^{c_0^2k(t-s)}d_s(x,y) \geq  d_t(x,y) \geq d_s(x,y) - \ga c_0\sqrt{s-t} \\
 && \vol(B_{g(t)}(x,1)) > \frac{v_0}{2}>0\\
 && \text{ \rm for all } \ x,y \in M, s\leq t, \ t,s \in  [0,S].
 \end{eqnarray*}
\end{thm}
\noindent {\it Proof idea.} 
We assume for the moment that $k=k(v_0)$ is sufficiently small. 
We define $c^2_0:= \frac{1}{\sqrt{k}},$ and note that then $2c^2_0k = 2\sqrt{k}.$ 
Using the previous theorems, we see that the volume and distance estimates and the  Ricci curvature lower bound  for $t\in (0,S)$  will be satisfied, for $S=S(c_0(v_0))$ small enough, as long as   
$\sup_M |\Rm(\cdot,t)|< \frac{c^2_0}{t}$ holds. 
Hence, in trying to show that   $\sup_M |\Rm(\cdot,t)|< \frac{c^2_0}{t}$ remains true we may assume that the volume and distance estimates are satisfied and that $\sup_M |\Rm(\cdot,t)|< \frac{c^2_0}{t}$ fails at a first time $t_0\leq S(v_0).$  
We note that \begin{eqnarray} \vol(B_{g(t)}(x,r)) \geq \frac{1}{4}v_0 r^3 \label{volley} \end{eqnarray}  will  hold for all $t \leq t_0\leq S(v_0)$ for any $x \in M$ for all $r\leq 1$ due to
the Bishop-Gromov-Volume-Comparison principle.  
A blow up argument, that is rescaling and translating the solution by $\ti g(\ti t):= C g(\frac {\ti t}{ C} + t_0),$ $C:= \frac{c^2_0}{t_0}$ will lead to a solution,  which contradicts the conclusions  of Perelman's Theorem, Theorem \ref{PerThm},  if $c_0 = c_0(v_0)>0$ is chosen  large enough (that is $k(v_0)>0$ is small enough), in view of the inequality \eqref{volley}.  The general case follows by some scaling arguments combined with the Bishop-Gromov volume comparison principle and applications of Theorem \ref{distance_and_volume_estimates_ricci_flow} 
:   See   \cite{SimonCrelle3D} for details.
$\Box$ 
 
A theorem of Gromov, says the following.
\begin{thm} (Gromov , Theorem 5.3 \cite{Grom})\label{GromThm}
Let $(M_i,g_i)$ be a sequence of smooth, complete, connected  Riemannian manifolds 
with $\Rc(g_i) \geq -k$ for some $k>0$. Then, after taking a subsequence,  there exists a pointed 
Gromov-Hausdorff limit  $(X,d_X,x) = \lim_{i\to \infty} (M,d_i=d(g_i),p_i)$ for any $p_i \in M_i.$ 
\end{thm}

Hence the class of initial Riemannian manifolds  appearing in Theorem
\ref{simon-main} has a sub-sequence which converges to a metric space $(X,d_X).$ Theorem \ref{simon-main} guarantees that there is a solution to Ricci flow coming out of $(X,d_X)$.

\begin{thm}\label{limit-theorem}
Let $(M^3_i,g_i(0),p_i),$ $i\in \N,$ be 
a sequence of smooth, connected, complete Riemannian three manifolds $(M^3_i,g_i(0),p_i)$ with bounded curvature and
$\Rc(g_i(0)) \geq -k_i$ and $vol(B_{g_i(0)}(x,1)) \geq v_0>0$ for all $i\in \N,$ where $k_i \in (0,\infty)$ and $k_i \to k\in [0,\infty).$ Then a subsequence of  $(M^3_i,g_i(0),p_i),$ converges to a possibly non-smooth metric space
$(X,d_X,x),$ and 
there exists a smooth, three dimensional, complete, connected   solution to Ricci flow, $(M^3,g(t))_{t\in (0,S)},$ and a metric $d_0$ defined on $M,$
such that $d_t \to d_0$   uniformly as $t\downto 0,$ $(M,d_t)$ has the same topology as $(M,d_0),$  and  $(M,d_0)$ is isometric to $(X,d_X)$.
In particular, $X$ with the topology given by $d_X$ admits a   $C^0$ manifold structure, and hence, due to a theorem of Moise,  admits a smooth   manifold structure.
The solution $(M,g(t))_{t\in (0,T)}$  satisfies the estimates in the conclusions of Theorem  \ref{simon-main}.
\end{thm}
Note: If $k_i<0,$ $k_i \to 0,$ then the conclusion for the curvature is
$\Rc(g(t)) \geq 0.$\\
{\it Proof sketch}.
The existence of $(X,d_X,x)$ follows from Gromov's Theorem, Theorem \ref{GromThm}. 
Let
$(M_i,g_i(t),p_i)_{t\in [0,S)}$ be the solutions coming from Theorem \ref{simon-main} with inital value $(M_i,g_i(0))$.  We can, after taking a subsequence,  take a smooth Cheeger-Hamilton Limit (see \cite{Ham-Com} )
for $t>0$ to obtain a smooth limiting solution
$(M,g(t) )_{t\in (0,S)},$ satisfying the estimates in the conclusion of Theorem \ref{simon-main}.
The distance estimates,   guarantee the existence of $d_0$,  and  that the topology of $(M,d_0)$ is the same as that of $(M,d_t)$ for all $t\in (0,S),$ and that $(M,d_0)$ is isometric to $(X,d_X)$.
$\Box.$

Examining the ingredients of Theorem \ref{simon-main},  we see that in order to apply a similar   procedure to other curvature conditions in other dimensions, one will need the following:\\
Ingredient 1) A theorem similar to the one of Perelman, Theorem \ref{PerThm}, for other curvature conditions.\\
Ingredient  2) Estimates for curvature bounds from below in the setting that $|\Rm(\cdot,t)|\leq \frac{c^2_0}{t}$.\\
Ingredient 3) The curvature bound from below being considered should imply a Ricci curvature bound from below, in order to make it possible to  use Theorem \ref{distance_and_volume_estimates_ricci_flow}. .\\ 
{ {\bf Notes }: \\
i)   $\Ricci \geq -k$ is  also used to get a limit at time zero via Gromov Hausdorff convergence.\\ 
ii) In the case that one assumes for example  non-negative curvature operator (that is a bound from below by zero), then it is possible to  extend these
 results and methods to higher dimensions, as was done in  Schulze/Simon \cite{Sch-Sim} : see also the paper of Deruelle  \cite{Der-Smo-Pos-Cur-Con}.  As we pointed out in the introduction, the curvature operator is known to remain non-negative (\cite{ham4d}). Also the Theorem of  Perelman, Theorem \ref{PerThm},  holds in this setting, so no further result is necessary in this case.\\
iii) In the case that the  one assumes  PIC2 then it is possible to obtain a similar result to this in all dimensions, as was shown by  \cite{CrWi}. In fact they do not need to assume that the initial manifold has bounded curvature. Using a Cheeger-Gromoll exhaustion function and a doubling procedure, they obtain a solution and results also in the setting of  curvature being initially unbounded.\\  
iv)   
In the case that we have merely a negative-bound from below at time zero, for example $\curlR(g(0)) +kI\geq 0$  for a $k>0$, 
attempts to generalise the above "Ingredient  2"  to higher dimensions by adding multiples of $t \Sc(g(t))I   $ and $I$    to $\curlR(g(t))$ for example, have, to this date,  not been successful,  except  in the case that
one already knows $|\Rm(g(t))|\leq \frac{c^2_0}{t}$ and  $c^2_0 \leq \ep(n),$ with $\ep(n)$ sufficiently small, as explained in   Section \ref{c0flowsection}.
A break through for the case that the bound from below is negative in higher dimensions was achieved by Bamler/Cabezas/Rivas \cite{Bam-Cab-Riv-Wil}. 
There they examine the quantity  $\ell(t)= \inf\{ \al \in [0,\infty)\  | \ \curlR(g(t)) +\al I \in \curlC \} $ where $\curlC$ represents the curvature condition (cone)  being considered. 
Due to some scaling arguments (as in the proof of Theorem \ref{simon-main} above), it suffices to consider the case that one has at time zero
$\ell(0) \leq \ep(n,v_0),$ where $\ep(n,v_0)>0$ is as small as we like.  They  show, that
the reaction equation  of  $\ell(t)$ satisfies an inequality of the type 
\begin{eqnarray} \label{estimate_ell}
\partt \ell(t) \leq  \Sc(g(t))\ell(t) + C_2\ell^2.
\end{eqnarray}
They are then able to prove upper bounds on the Gaussian corresponding  to the equation 
\begin{eqnarray} \label{modlap}
\partt u = \lap_g u + \Sc(g(t))u,  
\end{eqnarray}
in the setting that $|\Rm(g(t))|\leq \frac{c^2_0}{t}$ 
 which then allows them to obtain a bound on $\ell(t)$ by comparing $\ell$ to the solution $u$  to  \eqref{modlap}  whose initial value is constant,  $u(0):= \ep$.

The proof  of the estimate \eqref{estimate_ell} in \cite{Bam-Cab-Riv-Wil} is   very precise, and takes up a large part of the  paper.  There is no room for estimating by large constants, and   an estimate of the 
form $\partt \ell(t) \leq  C\Sc(g(t))\ell(t) + C_2\ell^2$ with $C>1$ in place of \eqref{estimate_ell} is {\bf not} sufficient for the arguments  used in the paper, that follow estimate \eqref{estimate_ell}. 
  
All the curvature conditions considered in the following theorem have the property that a bound on $\ell$ implies a bound from below on $\Ricci$. This means that Theorem \ref{distance_and_volume_estimates_ricci_flow} can be applied  and  in particular the volume estimates hold.
Furthermore, versions of Perelman's theorems  are proved for curvature conditions other than the curvature operator, which are
combined with  the volume bound from below, as in the proof of Theorem \ref{simon-main} explained above, to obtain that the curvature is bounded by $c^2_0/t$.   
\begin{thm}(Bamler-Cabezas-Rivas-Wilking:  \cite{Bam-Cab-Riv-Wil})\label{bamler-cabezas_rivas-wilking}
For $n\in\N$ and $v_0,k>0$ there exists a $c_0(n,v_0) $ and $S(n,\max(k,1), v_0)>0$ such that the following holds.  
Let $(X,d_X,x)$ be the pointed Gromov-Hausdorff limit of smooth, complete, connected  Riemannian manifolds $(M^n_i,g_i(0),p_i)$ with bounded curvature and assume that
\begin{eqnarray*}
&&\vol_{g_i(0)}(B_{g_i(0)}(x,1)) \geq v_0  \text{ for all  }  x\in M_i 
\end{eqnarray*}
 and assume further that  for some sequence $(k_i)_{i\in \N} \subseteq \R^+$ with $k_i \to k \in [0,\infty)$ as $i\to \infty$ we have
  $\curlR(g_i(0)) + k_i I \in \curlC,$   where $  \curlC$  is the curvature condition $\curlC_{CO}, \curlC_{2CO}, \curlC_{IC1}, 
  \curlC_{IC2}$ or the Riemannian manifolds are  K\"ahler with complex structure $J_i$ and $  \curlC = \curlC_{HB}.$
Then there is a Ricci flow $(M,g(t))_{t\in (0,S]}$ coming out of $(X,d_X)$ (a K\"ahler Ricci flow   $(M,g(t))_{t\in (0,S]}$  with complex structure $J$ in case  $\curlC = \curlC_{HB} $)    
such that 
\begin{eqnarray*}
&&|\Rm|(\cdot,t) \leq \frac{c^2_0}{t}\\
&&   \curlR(g(t)) + k c^2_0 I \in \curlC,\\
 &&    e^{-c^2_0 k \ga(n)(t-s)}d_s(x,y) \geq  d_t(x,y) \geq d_s(x,y) - \ga(n)c_0\sqrt{s-t} \\
 && \vol(B_{g(t)}(x,1)) > \frac{v_0}{2}>0\\
 && \text{ \rm for all } \ x,y \in M, s\leq t, \ t,s \in (0,S].
\end{eqnarray*}
\end{thm}
\noindent{\bf Notes}\\
i) Note that one can apply the Theorem to the case $(M_i,g_i(0)) = (M,g(0))$ : this leads to a  existence time, depending on $v_0$ and $k$ and estimates from above and below on the curvature,  as well as distance and volume control, all depending on $n,v_0, k$  as stated. \\ 
ii)  
The condition WPIC is also examined in the paper, and an estimate of the type, 
$\partt \ell(t) \leq  \Sc(g(t))\ell(t) + C_2\ell^2$  is proved in that setting.
However, in the setting of WPIC, a bound on $\ell\leq 1$ does not imply a lower bound on the Ricci curvature. Hence, one cannot use the results of Lemma \ref{distance_and_volume_estimates_ricci_flow} above. 

\begin{center}
{\bf Local Ricci flows and    local curvature 
conditions}
\end{center}
In some cases the initial manifold being considered may have some bound on curvature from below, but  
no  bound on the  norm of the curvature.
Hence even starting the Ricci flow  presents difficulties. 
We survey here some results on the existence of a Ricci flow in this setting. 
As a prototype case, we consider   a three  dimensional smooth, complete Riemannian manifold $(M^3,g(0))$ which is uniformly non-collapsed, that is 
$\vol(B_{g(t)}(x,1)) \geq v_0 >0$ for all $ x \in M,$ $\Rc(g(0)) \geq -\ep,$  $\ep \leq \ep(v_0),$ but we do not assume bounded curvature. That is
$\sup_M |\Rm(g(0))| = \infty$ is possible.

One method to obtain  existence of a Ricci flow on such a manifold, is to construct {\it local Ricci flows} on balls of radius $s>0.$ 
Uniform estimates are  also proved on interior regions of each of the local flows, which guarantee that a limiting solution  $(M,g(t))_{t\in [0,T)}$ 
 as   $s \to \infty$ exists and satisfies a bound    $ |\Rm(g(t))|\leq \frac{C_0(v_0)}{t}$ for all $t\in (0,T)$ and $\Rc(g(t)) \geq -C(v_0).$ 
In order to start the local flow, one   performs a conformal change of the metric near the boundary of the Ball $B_{g(0)}(x_0, s )$ being considered  as introduced by Hochard in \cite{hochard}.  
This  leads to a smooth, complete, connected  metric of bounded curvature, which is unchanged on the region $B_{g(0)}(x_0,r := s-\frac{1}{2})$.
One then flows this complete space for a short time $t_1$. Localising the estimates from the beginning of this section, 
leads to estimates of the type 
$|\Rm(g(t))| \leq \frac{C_0}{t}$  and
$\Rc(g(t))\geq -  1 $ for $t\leq t_1$ on $B_{g(0)}(x_0,r  -  C_1  \sqrt{t_1}).$
Now one performs a conformal change again on 
$B_{g(0)}(x_0,r  -  C_1  \sqrt{t_1})$ for the metric $g(t_1)$ and repeats the procedure. 
Continuing on in this way and   keeping careful track of the constants and the estimates proved, one is  able to obtain a local flow, which is defined for a time 
$T(v_0)>0$ on $B_{g(0)}(x_0, s-1 )$ and has  a lower bound on the Ricci curvature  $\Rc(g(t)) \geq -1$ and an upper bound of the form $|\Rm(t)|\leq \frac{c^2_0}{t}$  for all $t\in [0,T(v_0)]$ on $B_{g(0)}(x_0, s-2 ).$ 
Letting $s \to \infty$ leads to a solution as in Theorem \ref{simon-main}.
See the proof in Simon/Topping \cite{SiTo2} for the details. 
\begin{thm}(\cite{SiTo1}, \cite{SiTo2},\cite{hochard}) 
Theorem \ref{simon-main} and Theorem \ref{limit-theorem} are still correct, without the assumption that $(M,g(0))$ ($(M_i,g_i(0))$ ) has (have)  bounded curvature. 
\end{thm}

In the paper \cite{Bam-Cab-Riv-Wil} they showed,  using conformal changes of the original metric,   and taking a limit, that one can obtain a solution to Ricci flow starting with a uniformly non-collapsed manifold with $\curlR(g(0)) + k I \in \curlC_{CO}$ or $\curlR(g(0)) + k I \in  \curlC_{IC2},$ without the assumption that the norm of the curvature at time zero  of the manifolds in question are  bounded. 
In the Ph.D. thesis of Hochard \cite{HochardThesis}, and the  paper  of 
Lai 
\cite{Lai1}   it was shown that one can, using the estimates from \cite{Bam-Cab-Riv-Wil},  generalise the methods 
of \cite{SiTo2}, \cite{hochard} to obtain a Ricci flow for a uniformly non-collapsed Riemannian manifold with $\curlR(g(0)) + k I \in \curlC,$ and $\curlC = \curlC_{IC1}$ or $\curlC = \curlC_{2CO},$
and in  Lee/Tam \cite{LeeTam}, the authors  were able to combine   these methods with methods from K\"ahler geometry, to obtain a Ricci flow in the case that  $\curlR(g(0)) + k I \in \curlC_{HB}$. 
\begin{thm}(\cite{Lai1}, \cite{hochard},\cite{HochardThesis} \cite{LeeTam},\cite{Bam-Cab-Riv-Wil}) 
The results of Theorem \ref{bamler-cabezas_rivas-wilking} still hold
if we remove the assumption that the manifolds $(M_i,g_i(0))$ have
bounded curvature. 
\end{thm}
Notes: 
In the paper \cite{LeeTam} they considered an even weaker condition, namely
that the orthogonal holomorphic bisectional curvature and the Ricci curvature are bounded from below when showing that a Ricci flow exists: see
\cite{LeeTam} (the conditions they considered are implied by  the condition, that the holomorphic bisectional curvature is bounded from below) . 

In the process of producing the local Ricci flow, it was  noticed that the arguments are purely local. 
These methods were subsequently   pushed even further to obtain so called {\it Pyramid Ricci flows}, $g(t)$  which are defined on space time regions  $\Omega \subseteq M \times(0,T),$ 
$\Omega:= \cup_{i=1}^{\infty}B_{g(0)}(p,i)\times [0,T_i]$ 
where $T_1(v_0,k)>T_2(v_0,k)>T_3(v_0,k)\ldots$ and $T_i \to 0$ as $i \to \infty,$
 under the assumption that $\vol_{g(0)}(B_{g(0)}(p,1))\geq v_0,$
 and $\curlR(g(0)) + k I  \in \curlC$
 where $\curlC$ is one of the conditions  (i)-(iv) or (vi)-(ix), and $p\in M$ is a fixed point.  
As an application, one can prove for example,   that  any possible Gromov-Hausdorff pointed limit   of smooth three manifolds $(M^3_i,g_i(0),p_i)$  satisfying  $\vol_{g_i(0)}(B_{g_i(0)}(p_i,1))\geq v_0,$
 and $\Ricci(g_i(0)) + kg_i(0) \geq 0$  
   is a topological manifold :  Simon/Topping \cite{SiTo2}. 
See the following papers for further results of this type and  further information on this setting: 
\cite{SiTo2},  McLeod/Topping \cite{TopMc1}, McLeod/Topping \cite{TopMc2}, McLeod \cite{Mc}, 
\cite{HochardThesis}, \cite{LeeTam},
  (cf. Liu \cite{Liu1},\cite{Liu2}).

\begin{center} 
{\bf Flowing in a pinched setting }
\end{center}
Here we give a short overview of some results on Ricci-flowing Ricci-pinched  respectively  PIC1-pinched manifolds.\\
{\bf $(M^3,g(0))$ 3-dimensional and Ricci pinched}.\\
In the paper of Lott \cite{Lott-Ricci-pinched}, it is shown that if $(M^3,g(0))$ is smooth, connected , three-dimensional, complete, non-compact, has bounded curvature, and is Ricci pinched, non-flat, that is
\begin{eqnarray*}
&& \Rc(g(0)) \geq \si_0 \Sc(g(0)) g(0)\\
&& \Sc(g(0))(p)>0 \text{ at  (at least) one point } p 
\end{eqnarray*}
for some $\si_0>0$,    then there is a $\si_1, c_0>0$ and a Ricci flow solution 
$(M^3,g(t))_{t\in (0,\infty)}$ with  
\begin{eqnarray}\label{thestuff}
&& |\Rm(g(t))|\leq \frac{c^2_0}{t}\cr
&& \Rc(g(t)) \geq \si_1 \Sc(g(t)) g(t)>0\cr
&& AVR(g(t)) = V_0> 0 
\end{eqnarray}
 for all $t\in (0,\infty)$ for some $V_0>0$ where
 \begin{eqnarray*}
AVR(M,g):= \lim_{r \to \infty } \frac{\vol_{g}(B_{g}(x,r))}{r^3}.
\end{eqnarray*}
In particular $AVR(M,g(0))>0.$ 
  Blowing down this solution, that is considering the Cheeger-Hamilton limit, 
  $(Z,\ell(\hat t),x)_{\hat t \in (0,\infty)}= \lim_{j\to \infty} (M,\frac{1}{j}g(\hat t j),p)$ one obtains a solution to Ricci flow coming out of a  metric cone, where the cone is an Alexandrov space with non-negative sectional curvature (see Deruelle/Schulze/Simon \cite{DSS2} for details).  
In the paper \cite{DSS2}, the authors then show, that this solution must be close to an  expanding soliton coming out of the cone. This, combined with the fact that the solution is Ricci pinched,  then implies that the 
solution is in fact flat, and hence the original space $(M,g(0))$ is isometric to $(\R^3,\de)$ which contradicts  the assumption that $\Sc(g(0))(p)>0$ at (at least) one point $p$.
This then proves a conjecture  of Hamilton in the bounded curvature setting.  See \cite{DSS2} for the details. The conjecture  as stated in \cite{Lott-Ricci-pinched} is:\\
{\bf Hamilton Conjecture  (bounded curvature case) } : 
If  $(M,g(0))$ is smooth, connected, complete,  and   has bounded curvature, and is  Ricci pinched, then it is either Riemannian flat or  compact.\\
In the book  \cite{Chow-Lu-Ni}, the conjecture of Hamilton is stated  without the assumption of bounded curvature:\\ 
{\bf Hamilton Conjecture  (general  case)} :  
If  $(M,g(0))$ is smooth, connected, complete  and  Ricci pinched, then it is either Riemannian flat or  compact.\\ 
In the paper Lee/Topping \cite{Lee-Top-3d}, the authors  consider Ricci pinched three manifolds, which don't necessarily have bounded curvature. Using   methods touched upon 
above (local Ricci flows, local distance, local volume, local curvature estimates  ...), and other methods,  they show that   in the case that
$(M,g(0))$ is not assumed to have bounded curvature, that then there is a solution
$(M,g(t))_{t\in (0,\infty)}$  satisfying \eqref{thestuff}, and hence, using the result from the bounded curvature setting,  it follows that  the conjecture of Hamilton in the general setting,   
that is without the  assumption that the curvature is bounded, is correct. \\
\noindent {\bf $(M^n,g(0))$ PIC1 Pinched. }\\
Let $(M,g(0))$ be  smooth, connected, complete. 
We say $(M,g(0))$ is $\si_0>0$ PIC1 pinched, if  
$\curlR(g(0)) -\si_0 \Sc(g_0) I \in \curlC_{IC1}$.
In the paper of Lee/Topping \cite{Lee-Top-PIC2}, they show,
for non-compact $\si_0>0$ PIC1 pinched manifolds, using methods touched upon above 
(local Ricci flows, local distance, local volume, local curvature estimates,  ...), and other methods,  that there is a solution $(M,g(t))_{t\in (0, \infty)}$ 
starting with $(M,g(0))$  such that  $(M,g(t))$ is $\si_1>0$ PIC1 pinched,
and $|\Rm(g(t))|\leq \frac{c^2_0}{t}$ for all $t\in (0,\infty)$.

A second major result of the paper (\cite{Lee-Top-PIC2}) shows, that if we further assume that
$\curlR(g(0)) \in \curlC^{+}_{IC2},$ then in fact  
$(M,g(0))$ is flat.

In a subsequent paper  of Deruelle/Schulze/Simon, \cite{DSS3},     the authors investigate  the solution $(M,g(t))_{t\in (0,\infty)}$
of \cite{Lee-Top-PIC2} further, and show for initial data that is PIC1 pinched, that if    one further  assumes   $AVR(M,g(0)) >0$, without necessarily assuming  $\curlR(g(0)) \in \curlC_{IC2},$ that then $(M,g(0))$ is isometric to $(\R^n,\de)$.

We  {\bf conjecture } that the condition $(M,g(0))$ is complete, smooth, non-compact, connected,   non-flat, PIC1-pinched, 
will imply  $AVR(M,g(0)) >0$.
This, combined with the results of \cite{DSS3} on the solution of  \cite{Lee-Top-PIC2}, as  explained above,  would then show the generalised PIC1 Hamilton Conjecture  in any dimension:\\
{\bf    Generalised PIC1 Hamilton conjecture }:  $(M,g(0))$   smooth, connected, complete,    $\curlR(g(0)) -\si_0 \Sc(g_0) I \in \curlC_{IC1} \implies  (M,g(0))$ is flat or compact.\\

We note that Topping conjectured in \cite{TopConj} that any  complete, non-compact, smooth, connected  $n$- manifold which is $PIC1$ will be diffeomorphic to   Euclidean $n$ space.  
 
 \begin{center}
  {\bf Flowing in an almost Euclidean setting }
 \end{center}
In \cite{Per1},  Perelman proved his  {\it Pseudolocality Theorem}, Theorem 10.1 in  \cite{Per1},  for {\it almost euclidean} regions.
Various other notions of  {\it almost euclidean} and {\it Pseudolocality } type Theorems have since then been considered  and proved ( 
see for example   \cite{ChauTamYu} , \cite{WangY}, \cite{TW},  \cite{Wang2}, and the references in these papers and can be used to obtain some rough notions of weak solutions to Ricci flow starting with potentially  non-smooth almost euclidean data. 
As a prototype theorem in the almost euclidean  setting, we present  here a theorem of Chan/Chen/Lee, \cite{ChanChenLee}.
\begin{thm}(Chan/Chen/Lee, Theorem 1.4  in \cite{ChanChenLee}) \label{ChanChenLee}
For any positive integer $n\geq 3$ and constant $A\geq 1000n$, there exists a constant $\ep_0(n,A)>0$ such that the following holds. 
Suppose $(M_i^n,g_i(0),p_i)$ is   a pointed sequence of smooth, complete, connected Riemannian manifolds  with bounded curvature satisfying  the following properties:
\begin{enumerate}
\item  $\Rc(g_i(0)) \geq -\lambda,$
\item $(\int_{B_{g_i(0)}(q,2)} |\Rm(g_i(0))|^{\frac n 2} d\vol_{g_i(0)} )^{\frac 2 n} \leq \ep_0$ for all $q$ in $M_i,$ 
\item $\nu(B_{g_i(0)(q,5}),g_i(0),1) \geq -1$ for all $q$ in $M_i,$ 
\end{enumerate}
where $\nu$ is the localised entropy, a localisation of Perleman's entropy introduced in \cite{Wang}.
Then there exists  solutions $(M_i,g_i(t)),p_i)_{t\in [0,T)}$ for a uniform $T>0$ (independent of $i$) 
satisfying, $|\Rm(g_i(t)|\leq \frac{C(n,A) \ep_0}{t},$
$inj(M_i,g_i(t))   \geq \ga(n) \sqrt{t}$  for all $t\in (0,T)$ and, after taking a subsequence,  the  pointed Cheeger-Hamilton limit $(M,g(t),p)_{t\in (0,T)}$  of 
$(M_i,g_i(t)),p_i)_{t\in (0,T)}$   exists and also satisfies  $|\Rm(g(t)|\leq \frac{C(n,A) \ep_0}{t},$
$inj(M,g(t))   \geq \ga(n) \sqrt{t}$  and 
 $$(\int_{B_{g(t)}(q,2)} |\Rm(g(t)|^{\frac n 2} d\vol_{g(t)} )^{\frac 2 n} \leq C(n,A)\ep_0,$$
 for all $t\in (0,T).$ 
Furthermore 
a subsequence of $(M_i,g_i(0),p_i)$
 converges in the pointed Gromov-Hausdorff sense to a complete
metric space $(X,d_X,p)$,and there exists a complete metric space $(M,d_0)$ isometric to $(X,d_X),$ 
such that \\
  $\ga(n) (d_0(x,y))^{2} \geq   d(g(t))(x,y) \geq  d_0(x,y) -  \sqrt{\ep_0}\ga(n)\sqrt{t}$  and
$  d(g(t))(x,y) \geq  \ga(n) d_0(x,y)^{\frac 2 3}$ for all $x,y \in M$ with $d_0(x,y)\leq 1.$
Hence the topology of $(M,d_0),$ and   consequently $(X,d_X),$ is the same as that of $(M,d(g(t))$ for any $t>0$.  
\end{thm}

An interesting feature of this theorem  is the existence of the limit metric $d_0$ 
which is guaranteed by Lemma 2.4 of  
Huang/Kong/Rong/Xu \cite{HuanKongRongXu}  (which follows some unpublished notes of Bamler/Wilking , \cite{BamWil}) 
 as explained in the proof of Lemma 2.4  in 
\cite{HuanKongRongXu} and the proof of Theorem 1.4 in \cite{ChanChenLee}).
Although we see that 
$ \liminf d_t(x,y) \geq  d_0(x,y)$ for all $x,y \in M,$ it is not shown  in the proof of  of Theorem 1.4 in \cite{ChanChenLee}  that $d(t) \to d_0$ locally uniformly.

In Remark  1.3 of the paper \cite{ChanChenLee} the authors explain that the lower bound on the Ricci curvature in Theorem \ref{ChanChenLee} can be weakened to a lower bound on the scalar curvature, if one also includes a uniform volume doubling condition at some fixed  scale.


\section{Weakly differentiable Riemannian metrics and weak curvature lower bounds}\label{weakmetrics}
We saw in Section \ref{c0flowsection}, that one can transform the Ricci flow, using the harmonic heat flow, to the Ricci-DeTurck flow, which
is a parabolic partial differential system of equations.
One can then examine the solution using methods from the theory of parabolic partial differential systems of equations. This was carried out in
\cite{SimonC01} (see also \cite{SimonC02}) for $C^0$ metrics and metrics which are $1+ \ep(n)$-fair    to a smooth background metric $h$, with bounded curvature,  and in \cite{KochLamm1, KochLamm2} for similar settings.
In the paper of Lamm/Simon, \cite{LammSimon}, the case that the initial metric $g(0)$  is $W^{2,2}$ and $a$-fair to a smooth complete, metric $h$  with bounded curvature,  for an arbitrary $a>0$  was considered in the case $n=4$.  These  methods were used, generalised  and improved  to the  case of arbitrary dimension $n,$ assuming   $g(0)$ is in $W^{1,n}$ and $a$-fair to a smooth, complete, bounded curvature  metric $h$, for an arbitrary $a>0$, by Chu/Lee,  \cite{ChuLee}. 
In the paper of Jiang/Sheng/Zhang, \cite{JiShZh},  the authors considered the case that $g(0) \in W^{1,p}(M)$ with $M$ closed, smooth, and  $p>n.$ 
In the $W^{1,p}$ setting with $p>n,$ 
the Sobolev-embedding theorems show one that in fact $g(0)\in C^0(M),$ and so the paper of Simon \cite{SimonC01} (or Koch/Lamm \cite{KochLamm1}, \cite{KochLamm2})  can be used to prove existence of a flow, where $g(t) \to g(0)$ uniformly locally as $t\downto 0$.

In each of the papers \cite{LammSimon}, \cite{JiShZh}, \cite{ChuLee}, 
preservation of lower scalar curvature bounds is investigated. 
In the paper of \cite{LammSimon}, one considers the setting that  there is an $  0< a < \infty$ such that $(M,h)$ is a smooth, connected, complete Riemannian manifold 
such that  $ \sup_M |\grad^k \Riem(h)|\leq c(k) < \infty $ for each $k\in \N$  and $inj(M,h) \geq i_0> 0,$
$g(0) \in W_{loc}^{2,2}$ and 
\begin{eqnarray*}
&&\int_M |\grad^2 g(0)|^2 + |\grad g(0)|^2 d\vol_h   < \infty \\
&&  \frac{1}{a} h \leq g(0) \leq ah . 
\end{eqnarray*}

In the first part of the paper, one considers only the Ricci-DeTurck flow. 
As a first step, one notices by writing things in local Riemannian coordinates, that a local Sobolev inequality holds, and  that for any $\ep>0$, 
that after scaling the metrics $g(0)$ and $h$ by the same constant  (once) (we denote the scalings of the metrics  $g(0)$ and $h $ also by $g(0)$ and $h$)  we may assume that 

\begin{eqnarray}\label{scaleinvariant}
&&E(x):= \int_{B_1(x)}  |\grad^2 g(0)|^2 + |\grad g(0)|^4 d\vol_h  \leq \ep  \text{ for all } \ x \in M \cr
&&  \frac{1}{a} h \leq g(0) \leq ah   
\end{eqnarray}

The quantities in \eqref{scaleinvariant}  are scale invariant: if we scale  $g(0)$ and $h$ by  the same  constant $c\geq 1$,     and we denote the scalings of the metrics    also by $g(0)$ and $h$, then \eqref{scaleinvariant}  still holds.
This is used in many of  the scaling arguments in the paper.

After scaling $h$ and $g(0)$ once again, we may,   without loss of generality, also assume that  
\begin{eqnarray} \label{hassumptionsscaled}
 && (M,h) \ \ \mbox{is a smooth, connected, complete manifold  such that}\cr
&&\sup_M  {}^h|\gradh^i\Riem(h)|  < \infty \mbox{ for all } i\in \N_0 \cr 
&& \sum_{i=0}^4 \sup_M  {}^h|\gradh^i\Riem(h)|  \leq \de_0(a)  \\
 && inj(M,h)  \geq 100,\nonumber
 \end{eqnarray}
for a small positive constant $\de_0(a)$ of our choice.

Now one  mimics arguments from the theory of harmonic map heat flow (see for example \cite{Struwe}) and shows that a smallness of the energy 
\begin{eqnarray} 
&&  \int_{B_1(x)}  |\grad^2 g(t)|^2 + 
|\grad g(t)|^4 d\vol_h  \leq  \si(b, \ep),   \label{smallenergy} 
\end{eqnarray}
where $ \si(b, \ep) \to 0$ as $\ep\to 0$  
holds under Ricci flow for $t\leq T(n,b)$ as long as $  \frac{1}{b} h \leq g(t) \leq bh$ remains true, 
if $\int_{B_1(x)}  |\grad^2 g(0)|^2 + |\grad g(0)|^4 d\vol_h  \leq \ep $ holds at time zero. 
 
An important part of the argument is to show that indeed   $  \frac{1}{b} h \leq g(t) \leq bh$  remains true  for some $ b= b(a)$ if  
 $t\leq T(a)$  if the initial $\ep$ is chosen small enough. Assuming  this first   fails at time $t_0$ then one scales everything once again 
 by $\frac{1}{t_0}$, so that one may assume this first fails  at time $t_0=1.$ Standard parabolic arguments show that
 we must have 
 \begin{eqnarray}\label{gradients}
 |\grad^k g(t) |^2\leq \frac{c_k(b(a))}{t^k}. 
 \end{eqnarray} 
 Now using  that the energy  remains small, \eqref{smallenergy}, one sees  that  
 this implies 
 \begin{eqnarray*}
&& |\grad g(1)|^2\leq c(b(a),\ep)\\
&& \int_{B_{h}(x,2)} |g(1)-h|^2 \leq  c_1 a^2
 \end{eqnarray*}
 where $ c(b(a),\ep)  \to 0$ as $\ep \to 0,$ and the second  inequality is shown by integrating   $\partt \int_{B_{h}(x,2)} |g(1)-h|^2 $ from $t=0$ to $t=1$.
 But this tells us $|g(1)-h|^2  \leq c_2 a^2$  for $\ep$ small enough, which  leads to a contradiction if $b =b(a) \geq c_3a$, and $\ep = \ep(a)$ is chosen small enough.

The main result  obtained in \cite{LammSimon} for the Ricci DeTurck flow is  the following:
\begin{thm}\label{main1_start}
For any $1<a<\infty  $ 
 there exists a constant $\ep_1 =\ep_1(a) >0$ with the
following properties.
Let $(M^4,h)$  be a smooth four dimensional Riemannian manifold which satisfies  \eqref{hassumptionsscaled}. 
Let $g(0)$  be  a $W_{loc}^{2,2} \cap L^{\infty}$  Riemannian metric, not necessarily smooth, which satisfies 
\begin{align} 
& \ \ \frac 1 { a } h \leq g(0) \leq { a } h  \tag{a}  \label{a2}\\
& \ \ \int_{B_2(x)}  (|\gradh g(0)|^2 +   |\gradh^2 g(0)|^2)dh \leq \ep
\ \ \mbox{ \rm for all } \ \ x \in M  ,
 \tag{b} \label{b2}
 \end{align}
where $\ep \leq \ep_1.$ 
Then there exists a constant $T=T(a,\ep)>0$ and a smooth solution
$(g(t))_{t \in (0,T]}$ to \eqref{Meq} such that

\begin{align}
 & \ \ \ \ \frac{1}{400 a } h \leq g(t) \leq 400 a h  \tag{${\rm a}_t$} \label{a_t2}  \\
 &  \ \ \ \ \int_{B_1(x)} ( |\gradh g(\cdot,t)|^2 +   |\gradh^2 g(\cdot,t)|^2 )dh\leq 2\ep \ \  \tag{${\rm b}_t$}  \label{b_t2} \\
 & \ \ \ \ |\gradh^jg(\cdot,t)|^2 \leq \frac{c_j(h,a,\ep)}{t^{ j}}  \tag{${\rm c}_t$}  \label{c_t2} 
\end{align}
  for all  $x \in M  , \ t \in [0,T],$ 
where $c_j(h,\ep,a) \to 0$ as $\ep \to 0, $ and 
\begin{align}
 & \int_{B_1(x)} (|g(0) - g(t)|^2 + |\gradh(g(0) - g(t))|^2 + |\gradh^2(g(0) - g(t))|^2)dh \to 0  \tag{${ d}_t$}  \label{d_t2}  \\
& \mbox{ \rm as }  t\downto 0 \mbox{ \rm for all }  x \in M \nonumber
\end{align}
The solution is unique in the class of solutions which satisfy   \eqref{a_t2}, \eqref{b_t2}, \eqref{c_t2}, and \eqref{d_t2}.
The solution also satisfies the local estimates
\begin{align}
 &\sup_{x\in B_1(x_0)} |\gradh^jg(\cdot,t)|^2t^j   \to  0 \mbox{ for }  t \to 0  \tag{${\rm e}_t$} \label{e_t2}  
 \end{align}
 and for all $ 1<R\leq 2$ there exists a $V(a,R)>0$ such that 
 \begin{align}
 & \int_{B_1(x_0)}  (|\gradh g(\cdot,t)|^2 +   |\gradh^2 g(\cdot,t)|^2)dh \tag{${\rm f}_t$} \label{f_t2}   \\ 
 & \leq
 \int_{B_{R}(x_0)} ( |\gradh g(0)(\cdot)|^2 +   |\gradh^2 g(0)(\cdot)|^2 )dh + V(a,R)t  \nonumber
\end{align}
for all $x_0\in M$, $2\geq R>1$  for all $t \leq T.$
\end{thm}

In the  paper \cite{JiShZh}, the initial metric is assumed to be in $W^{1,p}$ with $p>n$. This enables the authors to prove  stronger time dependent estimates on the spatial gradients for solutions to the Ricci DeTurck flow, than those appearing in \eqref{gradients} :
  for example they show 
\begin{eqnarray*}
&&|\grad g|^2(t) \leq \frac{C}{t^{n/p}} \\
&& |\grad^2 g|^2(t) \leq \frac{C}{t^{\frac{n}{2p} + \frac 3 2 }} \\
\end{eqnarray*}
and we note that   $n/p <1$  and $\frac{n}{2p} + \frac{3}{2} < 2 $ : compare to  the case $k=1$ and $k=2$ in  \eqref{gradients}.

The next part of the paper  \cite{LammSimon}  is concerned with translating these results to the Ricci flow. That is one considers a 
{\it Ricci flow related solution}   $\ell(t):= (\Phi_t)^{*}g(t)$ where $\Phi_t$ is the family of diffeomorphisms solving the  
 the following ordinary differential equation:
 \begin{align}
    \partt {\Phi}^{\al}(x,t) =& V^{\al}(\Phi(x,t),t),  \ \ \ \text{for all}\ \  
    (x,t)\in M^n \times I,\nonumber \\
     \Phi(x,S)=& x.  \label{ODEDe2}
  \end{align}
  where $ {V}^{\al}(y,t) := -{g}^{\be \ga} \left({\up{g}
      \Gamma}^{\al}_{\be \ga} - {\up{h} \Gamma}^{\al}_{\be \ga}\right)
  (y,t)$

\begin{thm}\label{WeakRicciStart}
Let $1<a<\infty$ , $M= M^4$ be a four dimensional manifold, and $g(0)$ and $h$ satisfy  the assumptions \eqref{hassumptionsscaled}, \eqref{a2} and \eqref{b2}, with $\ep \leq \ep_1$ where $\ep_1= \ep_1(a)>0$ is the constant coming from Theorem \ref{main1_start}, and let 

$(M,g(t))_{t\in (0,T]}$  be the   smooth solution to  \eqref{Meq} 
constructed in Theorem \ref{main1_start}.
Then \begin{itemize}

\item[(i)] 
there exists a constant $c(a)$ and a smooth solution  $\Phi :M \times (0,T] \to M$ to   \eqref{ODEDe2} with 
$\Phi(T/2) = Id$ such that  $\Phi(t):= \Phi(\cdot,t) : M \to M$ is a diffeomorphism and 
 $d_h(\Phi(t)(x), \Phi(s)(x)) \leq c(a)  \sqrt {|t - s|}$ for all $x\in M.$  The metrics $\ell(t):= (\Phi(t))^*g(t), t\in (0,T]$ solve the  Ricci flow equation.
 Furthermore there are   well defined limit maps $\Phi(0): M \to M,$  $\Phi(0): = \lim_{t\downto 0} \Phi(t),$
 and $W(0): M \to M,$  $W(0): = \lim_{t\downto 0} W(t),$ where  $W(t)$ is the inverse of $\Phi(t)$ and  
  these limits  are obtained uniformly  on compact subsets, and  $\Phi(0), W(0)$  are  homeomorphisms inverse to one another.

 \item[(ii)] For the Ricci flow solution $\ell(t)$ from (i), there is a value $\ell_0(\cdot) = \lim_{t\downto 0} \ell (\cdot,t) $ well defined up to a set of measure zero, where the limit exists in the $L^p_{loc}$ sense, for any $p \in [1,\infty)$, such that, $\ell_0$ is positive definite, and  in $W_{loc}^{1,2}$ and for any $x_0 \in M$ and $0<s<t\leq T$  we have 
 \begin{eqnarray*}
 && \int_{B_{1}(x_0)}  |\ell(s)-\ell_0|^p_{\ell(t)} d\ell(t) \leq  c(g(0),h,p,x_0) s \cr
 &&  \int_{ B_{1}(x_0)}  |(\ell(0))^{-1}-(\ell(s))^{-1}|^p_{\ell(t)} d\ell(t) \leq c(g(0),h,p,x_0) |s |^{\frac 1 4}  \cr
&& \int_{B_{1}(x_0) )}|\gradg  \ell_0 |^2_{\ell(t)}   d\ell(t) \leq   c(g(0),h,p,x_0) t^{\si} \cr
 && \int_{B_{1}(x_0)} |\Rm(\ell)|^2(x,t) d\ell(x,t)   + \int_0^t \int_{B_{\ell(s)}(x_0,1)}  |\gradg \Rm(\ell)|^2 (x,s) d\ell(x,s) ds   \leq  c(g(0),h,p,x_0) \cr   
 && \sup_{B_{1}(x_0)}  |\gradg^j \Rc(\ell(t))|^2 t^{j+2}     \to 0  \mbox{ as } t \downto 0   \mbox{ for all } j\in \N_0 \cr
\end{eqnarray*}
for a universal constant $\si>0, $
where $\gradg$ refers to the gradient with respect to $\ell(t),$ $c(g(0),h,p,x_0)$ is a constant depending on 
 $g(0),h,p,x_0$ but not on $t$ or $s$.
  \item[(iii)]
 
The  limit maps $\Phi(0): M \to M,$  $\Phi(0): = \lim_{t\downto 0} \Phi(t),$
 and $W(0): M \to M,$  $W(0): = \lim_{t\downto 0} W(t),$ from (i) 
   are also  obtained in the $W_{loc}^{1,p}$ sense for $p\in [1,\infty)$.
 Furthermore,  for any smooth coordinates $\phi :U \to \R^n$, and $\psi: V \to \R^n$  with 
 $W(0)(V) \subsub U,$ 
the functions   $(\ell_0)_{i j} \of W(0):V \to \R $ are in $L_{loc}^p$ for all $p \in [1,\infty)$ and  
$(g(0))_{\al \be}: V \to \R $ and $(\ell_0)_{ij}: U \to \R$ are related by the identity 
$$(g(0))_{\al \be}  = D_{\al} (W(0))^i    D_{\be} (W(0))^j  ( ({\ell}_0)_{i j } \of W(0)),$$
which holds almost everywhere.
In particular: $\ell_0$ is isometric to $g(0)$ almost everywhere through the map $W(0)$ which is in $W_{loc}^{1,p},$
for all $p\in [1,\infty)$.

 \item[(iv)]
 We define 
  $d_t(x,y) = d(g(t))(x,y)$ and $ \ti d_t( p,q)  = d(\ell(t))(p,q)$,   for all  $x,y,p,q \in M,$  $t \in (0,T)$.
 There  are   well defined limit metrics $ d_0,\ti d_0: M \times M \to \R_0^+,$  $d_0(x,y) = \lim_{t\downto 0}  d_t(x,y) $, 
 and $\ti d_0 :=  M \times M \to \R_0^+,$  $\ti d_0(p,q) = \lim_{t\downto 0} \ti d_t(p,q),$ 
 and they satisfy $\ti d_0(x,y) = d_0(\Phi(0)(x),\Phi(0)(y)).$ That is, $ (M, \ti d_0)$ and $(M,d_0)$ are isometric to one another through the map $\Phi(0)$.\\ 
The  metric $d_0$ satisfies $d_0(x,y):= \liminf_{\ep \downto 0} \inf_{\ga \in C_{\ep,x,y}} L_{g(0)}(\ga),$ where 
$ C_{\ep, x,y} $ is the space of {\it $\ep$-approximative Lebesgue curves} between $x$ and $y$ with respect to $g(0)$: This space is defined/examined in Definition 8.2/Section 8 of \cite{LammSimon}.


\end{itemize}

 \end{thm}

\begin{rem} 
An attempt to  construct  a Ricci flow related solution $\ell(t)$ with $\Phi(0) = Id$ and $\ell(0) = \Phi(0)^*g(0) = g(0)$,  using  similar  methods  to those we  use to construct  the Ricci flow solution in Theorem \ref{WeakRicciStart},  could lead to a    Ricci flow solution,  which does not immediately become smooth. We avoid these {\bf non-smooth gauges} by choosing $\Phi$'s with $\Phi(\cdot,S) = Id$ for some positive time $S>0$ (not $S=0$) in   Theorem \ref{WeakRicciStart}.
See Remark 5.3 in \cite{LammSimon}  for more details.

\end{rem}

In the paper of \cite{LammSimon}  and \cite{JiShZh},    it is shown that when a scalar curvature bound from below, almost everywhere, exists for the weak starting metric $g(0)$ then this is preserved by the Ricci DeTurck flow and hence by the Ricci flow related solution. 
The method in Lamm/Simon involves showing that the weak quantity
$\int_{M} ((\Sc(g(t))-k)_{-})^2 $ remains zero when evolved by the Ricci DeTurck flow, if it is zero at time zero. 
Notice here that when $g(0) \in W^{2,2}_{loc}$ and $\frac{1}{a}h \leq g(0)  \leq ah$ that then the weakly defined quantity
$\Sc(g(0))$ is in $L^2 $ as is the quantity $(\Sc(g(0))-k)_{-}$ and hence $\int_{M} ((\Sc(g(t))-k)_{-})^2 $  is also well defined at time $t=0$.

In the paper of  \cite{JiShZh},   the argument is quite subtle and they show that
the Scalar curvature remains bounded from below  in the distributional sense, if $g(0) \in W^{1,p}$. One of the interesting features there, is that one needs to evolve
the test functions  $\phi$  used when testing  whether   $\Sc(g(t))-k \geq 0$ in the distributional definition of $\Sc(g(t)) \geq k.$  
  
The paper \cite{ChuLee} takes a slightly different approach and uses  a similar argument to the one given by Burkhardt-Guim, \cite{Burk-Guim1, Burk-Guim2} . They  consider  sequences of smooth metrics $g_{i}(0) \to g(0)$ in $W^{1,n}$ as $i\to \infty$ with $\Sc(g_{i}(0))\geq k$ and $g_{i}(0)$ $a$-fair to a fixed smooth metric $h$ (as above). This  then guarantees,  
that the solutions $g_{i}(t)$ with initial value $g_i(0) $ will have $\Sc(g_i(t)) \geq k$ and hence the limiting flow $g(t)$ (which  they show is well defined)  coming out of the weak metric $g(0) \in W^{1,n}$ will have  $\Sc(g(t)) \geq k.$ See \cite{ChuLee} for details.

\noindent{\bf Metrics with lower dimensional singular sets whose scalar curvature is bounded from below on  their  regular sets.}\\
After the first version of this survey was finished, some further results appeared  on smoothing metrics which are smooth away from a  singular set.
   
Burkhardt-Guim  \cite{Burk-Guim3} 
  considers the case that the singular set $S$ has dimension less than 
$n-2-\alpha,$ for some $\alpha>0$ and the metric  $g_0$ is defined on $\R^n$ and is $L^{\infty}$ close to the standard metric, in the sense that $|g_0-\de_0|_{\de_0} \leq \ep$ for an appropriately chosen $\ep$. 
If the scalar curvature is non-negative away from the singular set, then Theorem 1.1 of \cite{Burk-Guim3}
shows, using the Ricci flow and Ricci DeTurck flow, that there are smooth metrics $g_t, t\in (0,1)$ such that $g_t \to g$ in $C^{\infty}(\R^n \backslash S)$ as $t\to 0,$ with $\Sc(g_t) \geq 0$.

In the paper  \cite{LL24},   Lee and Liu consider the case that 
  $(M,h)$ is complete, connected and has  bounded curvature/ covariant derivatives of curvature, and    $g_0$ is a metric in   $W_{loc}^{1,2}$ which  satisfies $ \Lambda^{-1}h \leq g_0 \leq \Lambda h$ for some $\Lambda >0$  and the   Morrey type condition
$\int_{B_h(x_0,r)} |\grad g_0|^2 dh \leq Lr^{n-2 +\de}$   for all $ r\leq 1$ for all $x_0 \in M.$ 
  Proposition 4.2  and Theorem   4.6 of \cite{LL24} then show, that if the scalar curvature of $g_0$ is not less than $a$ away from a singular set  $S$  of co-dimension less than $2+\alpha$ for  some $\alpha>0,$ then a Ricci-DeTurck flow $(g_t)_{t \in [0,T)}$ exists and is smooth for all $t>0$ and has $\Sc(g_t) \geq a$ and $g_t \to g_0$ in $C^0_{loc}$ and $g_t \to g_0$ smoothly away from $S$. See Chapter 4 of \cite{LL24} for details and further applications to smoothing metrics while  keeping scalar curvature bounds from below, and Definition 4.1 of \cite{LL24} for the precise definition of {\it co-dimension} in this setting.

\section{Further weak flows}

In the papers of Kleiner/Lott \cite{KleinerLott1}, \cite{KleinerLott2}, the authors produced and studied  a flow in three dimensions, which can be continued past singularities, the so called {\it Singular  Ricci flow}.
It used  Perleman's/Hamilton's  work (\cite{Per1,Per2, HamFor}) on describing Ricci flows and their behaviour near possibly singular  times. 
The study of the Singular Ricci flow  was continued in the papers  of Bamler/Kleiner  \cite{BamKle1,BamKle2,BamKle3}, and \cite{Lai2}. 
See \cite{BamSurv} for an overview of some of the ideas/results/methods of this theory. 

In the paper \cite{Lai2}, Lai  studied this solution in the context of an   initial non-compact, complete, connected, smooth  Riemannian manifold with initial data being complete and having non-negative Ricci curvature and potentially unbounded norm of curvature. 
With the help of the theory of singular flows,  she showed that a smooth solution    to Ricci flow  starting with   initial values of this type exist. Whether
the solution is complete for all times larger than zero remains an open problem.

Haselhofer \cite{Has} continued the study of  Singular Ricci flows in higher dimensions, under the assumption that a canonical neighbourhood type theorem (see \cite{BamKle1}, \cite{Per1, Per2})  holds  : See \cite{Has} for details.


In the papers \cite{Bam20a,Bam20b, Bam20c} the author considers a class   of flows called {\it metric flows}. 
These occur as limits, for example,  of sequences of Ricci flows
$(M_i,g_i(t))_{t\in (-T_i,0]}$  whose space times $\curlX_i := M_i  \times(-T_i,0]$ do not degenerate in the limit, that is $ T_{\infty}:= \lim_{i\to \infty}T_i ,$ has $T_{\infty}  >0$, and for which 
$\curlN_{x_i,0}(\tau_0) \geq -Y_0 > -\infty,$ where  $\curlN$ is the pointed Nash-Entropy (see  \cite{Bam20a} for more details), as introduced in Section 5 of Perelman's paper \cite{Per1} (the $\log(Z)$ function) and the paper \cite{HN} of Hein/Naber. The limit is described as a {\it metric flow}, and has a regular part $\curlR$ and a singular part $\curlS.$ On $\curlR$ the {\it metric flow} corresponds to a Ricci flow. The parabolic{*}-Minkowski dimension   of the singular space is $\leq n-2$ (see Definition 3.43 of \cite{Bam20b} for more details).
In the case that the limit can be written as a smooth Ricci flow,
then the limiting {\it metric  flow} has $\curlS = \emptyset.$
The limiting {\it metric flow} structure is  given by  ($I = (- T_{\infty},0],)$
$$(\curlX,   \curlt  , (d_t)_{t\in I},  (\nu_{x;s})_{x \in \curlX,s\in I,  s\leq \curlt(x)}),$$
where $\curlt:\curlX \to \R$ is a time function defined on the limiting space-time $\curlX$,  $(\nu_{x;s})\in  \curlP(\curlX_s)$ are  probability measures, $d_t$ is a complete, separable metric defined on the {\it time slices} $\curlX_t = \curlt^{-1}(t)$. In the limiting situation we considered above, 
$d_t$ are the limit of $d(g_i(t))$ and $\nu_{x;s}$ the limits of
 $\nu(i)_{x;s}$
 which are  weightings of  the volume forms with the conjugate heat kernel measured at 
$x=(p,t),$ that is
$$ d\nu(i)_{(p,t);s}(q) = K(g_i)(p,t;q,s)d\vol_{g_i(s)}(q),$$
where  $K(g_i)(p,t;q,s)$ is the heat kernel of the Ricci flow.
See Section 3.7 of \cite{Bam20b} for details.  
The space of {\it metric flows} considered by Bamler are quite general, the topology of the time slices is not necessarily constant in time, and for example, the singular flow of Kleiner/Lott can be seen as a {\it metric flow}  as shown by Bamler in Section 3 of \cite{Bam20b}.  
The metrics and the probability measures
are required to satisfy certain compatibility conditions, which always hold on a super Ricci flow
and are independent of its dimension.
See \cite{Bam-Video} and \cite{BamSurv} for an overview of some of the ideas/results/methods  of this  theory. 

Haslhofer/Naber defined a weak Ricci flow using notions from  Stochastic Calculus. 
In \cite{Has-Nab} 
they showed that a smooth family of Riemannian metrics
$(g(t))_{t\in I}$ evolves by Ricci flow if and only if
a sharp infinite dimensional gradient estimate 
$|\grad_p \E_{(p,t)}[F]| \leq \E_{(p,t)}[|\grad^{\|}F|]$ holds for  a certain class of functions $F$ in the path space of its space-time.
Here $\E_{(p,t)}$denotes the expectation with respect to the Wiener measure of Brownian motion starting at $(p,t)$ and $\grad^{ \|}$
denotes the parallel gradient transport, which is defined via a suitable stochastic parallel transport. The authors proposed this inequality be used to define a weak Ricci flow.
See \cite{Has-Nab} for details.

In the paper \cite{Has-Choi}, 
the authors show   that every non-collapsed {\it metric flow} limit of Ricci flows, as provided by
Bamler's precompactness theorem from the paper \cite{Bam20b}, as well as any  singular Ricci flow in the sense of   Kleiner-Lott, \cite{KleinerLott1},  is a weak solution in the sense of Haslhofer-Naber, \cite{Has-Nab}.
See \cite{Has-Choi} for details.

Sturm considered weak (super)  Ricci flows in the context of 
optimal transport in his paper \cite{sturm2018super}. Munn and Lakzian, \cite{LakzianMunn} looked further at  weak  Ricci flows in such a setting and in particular   Ricci flows  that  can be   extended (weakly)  through neck-pinch     singularities. 
See also \cite{topping2009L},\cite{topping2009ricci}, \cite{McTop} for earlier   related works, where   (Super) Ricci flows are studied in this context.

\section{Open problems}\label{open-problems}

\begin{itemize}
\item[(a)]
Is there an estimate of the type 
$|\Rm(\cdot,t)|\leq \frac{c^2_0}{t}$ if one starts with manifolds which are uniformly non-collapsed and have positive isotropic curvature?
\item[(b)] 
Prove versions of the Theorem \ref{main1_start} of    Lamm/Simon in the $n$-dimensional case: $g_0\in W^{2,\frac n 2}$ and $h$ is $a$-fair to $g(0),$ $h$ smooth, complete, connected with bounded geometry. 
\item[(c)]
Is it possible to prove that the non-negativity of the curvature conditions discussed at the beginning of the paper are preserved along the weak flow of Lamm/Simon? That is,  can one prove that  the condition   $\curlR    \in  \curlC$ at time zero in a weak sense, where $\curlC$ is  one of the curvature conditions  \eqref{curvature-conditions}, continues to hold along the flow constructed by Lamm/Simon ?
\item[(d)]  
Is it possible to prove    versions of the theorems of Simon \cite{SimonCrelle3D},  Bamler/Cabezas-Rivas/Wilking \cite{Bam-Cab-Riv-Wil}, in the setting of Lamm/Simon, Theorem \ref{main1_start} : If 
  $\curlR + \ep I  \in  \curlC$ for some weak metric  $g(0)\in W_{glob}^{2,\frac n 2}$ and $h$ is $a$-fair to $g(0)$,   then is 
   $\curlR(g(t)) + \ep C I  \in  \curlC$ for $t\in [0,S)$ for some uniform $S>0$ depending on $h$ for the Lamm/Simon, \cite{LammSimon}, solution ?
   For general $n\neq 4$ it would first be necessary to solve (b), before considering this question.  
 \item[(e)]  Consider the questions of (c) or (d) in the case that $g(0)\in W_{glob}^{1,n}$ and $g(0) $ is $a$-fair to $h,$ $h$ is $a$-fair to $g(0),$ $h$ smooth, complete, connected with bounded curvature,    and $g(t)$ is the solution constructed  by Chu/Lee \cite{ChuLee}, where  here curvature should be defined in some distributional or some other weak sense.   
\item[(f)] Consider the questions of (c) or (d),  
 in the case that $g(0)\in W^{1,p}(M),$ $M$ closed, smooth, $p>n$  and $g(t)$ is the solution constructed  by
\cite{JiShZh}. Here curvature should be defined in some distributional or some other weak sense.
\item[(g)] 
Let $(X,d)$ be an Alexandrov space of dimension $n =  3 $ with curvature 
not less than $-1,$ which is homeomorphic to a manifold.  Can one find a smooth Ricci flow $(M,g(t))_{t\in (0,T)}$  coming out of $(X,d)$ , that is with $(M,d(g(t))) \to (M,d_0)$ as $t\downto 0,$ $(M,d_0)$ being isometric to $(X,d_X)$ such that $\sec(g(t)) \geq -K$ for some $K\in \R_0^{+}$ (see notes below)?
\item[(h)] 
Assume that $(X,d)$ is an RCD  (${\rm RCD}^{*}$) space   , with   dimension $n\leq  3 $ (or some other   dimension not more than/equal to three)  with Ricci curvature 
not less than $-1,$ in the RCD (${\rm RCD}^{*}$) sense (see the papers  \cite{Gig}, \cite{Cav-Mil} and references therein), which is homeomorphic to a three dimensional manifold. 
Can one find a smooth Ricci flow $(M,g(t))_{t\in (0,T)}$  coming out of $(X,d)$ , that is with $(M,d(g(t))) \to (M,d_0)$ as $t\downto 0,$ $(M,d_0)$ being isometric to $(X,d_X)$ such that $\Ricci(g(t)) \geq -K^2$ for some $K\in \R $? \\
\item[(i)] Let $(M^3,g(0))$ be smooth, complete, connected,  with non-negative Ricci curvature, possibly collapsed at infinity.
 Lai, \cite{Lai2}, produced a smooth Ricci flow $(M^3,g(t))_{t\in [0,T)}$ starting with this initial data. However it is unknown till this date, whether$(M,g(t))$ is complete for $t>0$.
Problems: \\
Is the solution of  Lai  complete at positive  times? \\
Is there some   smooth solution   to Ricci flow with initial value  $(M,g(0))$ which is complete at positive  times?\\
\item[(j)] Let $(M,g(0))$ be smooth, complete, connected, 
with $\curlR(g(0)) -\si \Sc(g(0)) \in \curlC_{IC1}$ for some $\si >0.$ 
Prove the {\bf Generalised PIC1 conjecture of Hamilton:} $(M,g(0))$ is compact or flat.
 \item[(k)] 
 Can one  remove the condition $inj(M,h)>0$ in Theorem \ref{C0Ricciflow-Thm} of Lamm/Simon? Strongly related to this    question is the question :  can we remove the condition $ inj(M,g_i(0))>0 $ from  Theorem \ref{Schlichting} of this paper? 
 \item[(l)] Can one remove the condition $inj(M,h)>0$ in Theorems \ref{main1_start} and \ref{WeakRicciStart} of Lamm/Simon?  
 \item[(m)] Show in the theorem of Chan/Chen/Lee,  Theorem \ref{ChanChenLee}, that $d(g(t)) \to d_0$ locally uniformly as $t\downto 0.$ 
 \item[(n)] Which classes of low regularity metric spaces,  which are limits of spaces with scalar curvature uniformly bounded from below,  can be evolved by the Ricci flow? 
 \item[(o)] Which existing theorems on Ricci flow with non-smooth initial data can be generalised to Ricci flow with boundary? What boundary conditions are necessary for theses generalisations?
 
\end{itemize}
{\bf Notes on the open problems}:
\begin{enumerate}
\item  If one knew, in (g), or (h), that $(X,d_X)$ is compact, $3$ dimensional,   and is given by the Gromov-Hausdorff  limit of smooth, closed,  Riemannian manifolds with a uniform lower bound  on the sectional curvatures or Ricci curvatures,   then one could use Section  \ref{sec-ricci-flow-non-collapsed-non-neg} to construct such a solution.  \\
\item  If one knew, in (g) or (h), that $(X,d_X)$  is almost euclidean everywhere, one could possibly  construct a sort of  flow $(M,g(t))_{t\in (0,T)}$ coming out of $(X,d_X),$ using the Pseudolocality Theorem of Perelman \cite{Per1} : But there would be  no guarantee, to this date,  that $\sec(g(t)) \geq -K$ or $\Rc(g(t)) \geq -Kg(t)$ holds, for some $K\in \R^{+}$, for all $t\in (0,S)$. \\
 
\item  In case $n=2$, and $(X,d_X)$ is a closed  RCD space then a theorem of Lytcheck/Stadler \cite{LytSta}, shows that $(X,d_X)$  is an Alexandrov space of dimension two. Alexandrov theory then shows that 
there is a sequence of smooth Riemannian spaces $(M_i,g_i(0)$ with $\sec(g_i(0)) \ge -k$ which converge to  $(X,d_X)$  in the Gromov-Hausdorff sense as $i \to \infty.$
Hence, using 
the results from Section \ref{sec-ricci-flow-non-collapsed-non-neg}, there is a solution  $(M^2,g(t))_{t\in (0,T)}$ coming out of $(X,d_X)$. 
See  the paper of Richard \cite{Richard-T-Alex} for more  details and references to  theorems on  Alexandrov spaces required in this procedure.  
\item   In (j), the conjecture doesn't involve a Ricci flow.
However: In the three dimensional case, in \cite{Lott-Ricci-pinched}, Lott used   in the  non-flat, non-compact case, the Ricci pinched Ricci flow solution that  he constructed, and the estimates satisfied by it, in a fundamental way when showing   that $AVR(M,g(0))>0$.
 \item 
With respect to (i):  in the paper \cite{ChauMart} the authors showed that  if one  assumes a  collapsing of the type  
$\frac{\vol(B_{g(0)}(x,r))}{r^3}  \geq f(r)$ for a particular function $f:\R^+ \to \R^+$  where $\lim_{r\to \infty} f(r) = 0,$  then there is a complete solution to Ricci flow for a short time: See \cite{ChauMart} for details.
 
\item  Regarding  (k) and (l): The condition $inj(M,h)>0$ was not assumed in the paper of \cite{ChuLee}.
 \item  Regarding (n). \\
To flow such spaces it most likely will be  useful to understand more about  the regularity of  limits of spaces with scalar  curvature bounded from below, which is a field in itself. 
We refer the reader to the books edited by Gromov/Lawson  \cite{PerSc1}, \cite{PerSc2}, and the  
surveys   of Gromov \cite{GromFour}, Sormani/Tian/Wang \cite{tian2024oberwolfach}, 
Allen/  Bryden/Kazaras/Perales/Sormani  \cite{Allen},  Sormani \cite{Sormani},   the papers  of Lott \cite{LottSc}, \break   Munteanu / Wang \cite{MuntWang},
  Cecchini/Hanke/Schick\cite{CecHanSch}, Cecchini/Zeidler \cite{CecZei}, Cecchini/Frenck/Zeidler \cite{CecFrZei}, 
  B\"ar/ Hanke \cite{BaeHan}, Lee/Tam \cite{LeeSc}, as  well as the references/ articles  in all of these books/ papers/ surveys. The list of references we have given here is merely a starting list and is   by no means complete.  \\
  It is possible that flowing such spaces may lead to new insights into the regularity of such spaces. For example Ricci flowing the limits of    uniformly non-collapsed three manifolds   with Ricci curvature bounded from below, showed that such spaces are  indeed manifolds (see \cite{SimonCrelle3D}, \cite{SiTo2}). 
  \item regarding (o) : Some results already exist: see for example the references to Ricci flow with boundary given at the beginning of this paper.
\end{enumerate}

\end{document}